\documentclass[a4paper,11pt]{article}

\usepackage{amsmath,amssymb,amsthm,array,mathrsfs}
\usepackage{amssymb,latexsym, bbm,comment,epsfig}
\usepackage{booktabs}
\usepackage{pgfplots}
\usepackage{authblk} 
\usepgfplotslibrary{fillbetween}

\renewcommand{\theequation}                            %counter according section
       {\mbox{\arabic{section}.\arabic{equation}}}

\newcommand{\origsetminus}{} \let\origsetminus=\setminus           % redefine set difference
\renewcommand{\setminus}{\!\origsetminus\!}

{\theoremstyle{plain}
\newtheorem{definition}{Definition}[section]
\newtheorem{lemma}[definition]{Lemma}
\newtheorem{theorem}[definition]{Theorem}
\newtheorem{corollary}[definition]{Corollary}
\newtheorem{proposition}[definition]{Proposition}
 
} {\theoremstyle{definition}
\newtheorem{example}[definition]{Example}
}
\newtheorem{remark}[definition]{Remark}

\renewcommand{\mathbb}{\mathbbm}                     % use mathbbm
\renewcommand{\epsilon}{\varepsilon}                 % AMS symbols
\renewcommand{\phi}{\varphi}
\renewcommand{\theta}{\vartheta}
\renewcommand{\le}{\leqslant}
\renewcommand{\ge}{\geqslant}

\newcommand{\origfoo}{} \let\origfoo=\sqrt           % redefine square root
\renewcommand{\sqrt}[1]{\origfoo{#1}\;}

\renewcommand{\O}{{\mathcal O}}                      % big Landau
\newcommand{\abs}[1]{\left\lvert #1 \right\rvert}    % absolut value
\newcommand{\norm}[1]{\left\lVert #1 \right\rVert}   % norm
\DeclareMathOperator{\R}{{\mathbb R}}                % reals
\DeclareMathOperator{\Rp}{{\mathbb R}_+}             % positive reals
\DeclareMathOperator{\C}{{\mathbb C}}                % complexe values
\DeclareMathOperator{\N}{{\mathbb N}}                % integer

\DeclareMathOperator{\Z}{{\mathbb Z}}                % integer
\newcommand{\Zp}{\mathbb{Z}_{+}}
\newcommand{\A}{{\mathcal A}}

\DeclareMathOperator{\Borel}{{\mathcal B}}
\renewcommand{\S}{{\mathcal S}}

\newcommand{\scapro}[2]{\langle #1,#2\rangle}       %Skalarprodukt
\newcommand{\Scapro}[2]{[ #1,#2 ]}       %For the L2 product
\DeclareMathOperator{\1}{\mathbbm 1}
\newcommand{\Besov}[3]{B^{#1}_{#2,#3}(\R^d)}	%Besov space
\newcommand{\besov}[3]{B^{#1}_{#2,#3}}	%Besov space

\renewcommand{\d}{{\mathrm d}}
\DeclareMathOperator{\leb}{{\rm leb}}

\newcommand\restr[2]{{% we make the whole thing an ordinary symbol
  \left.\kern-\nulldelimiterspace % automatically resize the bar with \right
  #1 % the function
  \vphantom{\big|} % pretend it's a little taller at normal size
  \right|_{#2} % this is the delimiter
  }}

\allowdisplaybreaks

\usepackage[normalem]{ulem}
\usepackage{color}

\title{Regularisation of cylindrical L\'evy processes in Besov spaces}
\author{Matthew Griffiths}
\author{Markus Riedle}
\affil{Department of Mathematics \\ King's College  \\ London WC2R 2LS\\ United Kingdom}

\begin{document}
\maketitle
\begin{abstract}
In this work, we quantify the irregularity of a given cylindrical L\'evy process $L$ in $L^2(\R^d)$ by
determining the range of weighted Besov spaces $B$ in which $L$ has a regularised version $Y$, that is a 
stochastic process $Y$ in the classical sense with values in $B$. Our approach is based on characterising L\'evy measures on Besov spaces. As a by-product, we determine those Besov spaces $B$ for which the embedding of $L^2(\R^d)$ into $B$
is $0$-Radonifying and $p$-Radonifying  for $p>1$. 
\end{abstract}
\noindent
{\bf AMS 2010 Subject Classification:}  60G20, 47B10, 60H25, 60G51, 60E07. \\
{\bf Keywords and Phrases:}  cylindrical processes; generalised processes; Radonifying operators; regularisation;  Besov spaces. \\

\section{Introduction}
Cylindrical L\'evy processes are a natural generalisation of cylindrical Brownian motions or equivalently of Gaussian space-time white noises. Being cylindrical processes, they generally do not attain values in the underlying space. This fact may cause some surprising phenomena such as highly irregular paths of solutions for linear evolution equations driven by a cylindrical L\'evy process $L$, which is observed for example by Brze\'zniak and co-authors in  \cite{Brzezniak-etal} or Priola and Zabczyk in \cite{Priola2011}, and which is related to the cylindrical distribution of $L$ by Kumar and one of us in \cite{Kumar2018}. 

In this work, we initiate a new type of research question by quantifying the irregularity of cylindrical L\'evy processes in terms of its cylindrical distribution. For this purpose, we determine the range of Besov spaces $B$ in which a given cylindrical L\'evy process $L$ in the Hilbert space  $L^2(\R^d)$ of square integrable functions becomes a L\'evy process $Y$ in the classical sense, i.e.\ a stochastic process attaining values in the Besov space $B$. In this case, the classical L\'evy process $Y$ is called the {\em regularised} version of $L$ in $B$, a notion introduced by It\^o and Nawata in the work \cite{Ito-Nawata}. 
This approach is motivated by the following observation: by embedding the Hilbert space  $L^2(\R^d)$ into the space $S^\ast(\R^d)$ of tempered distributions, each cylindrical L\'evy process in $L^2(\R^d)$ becomes a classical L\'evy process in $S^\ast(\R^d)$ due to Minlos' theorem; see Fonseca-Mora \cite{Fonseca-Mora2018}. Besov spaces lie between $L^2(\R^d)$ and $S^\ast(\R^d)$, and thus offer a natural and sufficiently fine scale to quantify the irregularity of a cylindrical L\'evy process $L$ by determining those Besov spaces in which $L$ becomes a classical L\'evy process if they exist.

Besov spaces are a natural extension of H\"older-Zygmund and fractional Sobolev spaces, see Chapter 1 of the monograph \cite{Triebel2006} by Triebel for a comprehensive introduction, and have been extensively applied to measure smoothness of functions, e.g.\ of solutions of partial differential equations. In the probabilistic setting, they have been used to analyse regularity of sample paths for finite- and infinite dimensional stochastic processes, starting with the publications   \cite{Herren1997} by Herren and \cite{Schilling1997} by Schilling. However, our investigation is fundamentally different as we apply Besov spaces as a scale of spaces between $L^2(\R^d)$ and $S^\ast(\R^d)$  in which a cylindrical random variable in $L^2(\R^d)$ may become a classical random variable.

The line of research closest to ours originates from the consideration of image processing and appears in the publications  \cite{Aziznejad2018,Dalang2015,Fageot2016,Fageot2017a} by Dalang, Unser, Fageot and co-authors, in which the authors investigate different aspects of a L\'evy-type model in $S^\ast(\R^d)$. In order to characterise local smoothness and the asymptotic growth in \cite{Aziznejad2018}, 
they determine the Besov spaces in which the L\'evy-type model attains values almost surely. Although this result is related to our investigation, it only applies to the subclass of cylindrical L\'evy processes in $L^2(\R^d)$ which corresponds to the considered L\'evy-type model in $S^\ast(\R)$; this subclass is shown to be very special  in our previous work \cite{Griffiths2019}, requiring the corresponding cylindrical L\'evy processes to be stationary in time {\em and} in space. 

In order to derive results on the regularisation of cylindrical L\'evy processes in Besov spaces,
our first task is to characterise the L\'evy measures in  Besov spaces.
In most Banach spaces apart from Hilbert spaces, an explicit characterisation of L\'evy measures is not known. One of the exceptions are L\'evy measures on the sequence spaces $\ell^p$ due to a result by Yurinskii in \cite{Yurinskii1974}. 
%This result in general gives separate conditions for necessity and sufficiency of a $\sigma$-finite measure on a separable Banach space to form a L\'evy measure. 
Using the wavelet characterisation of  Besov spaces, the result by Yurinskii will enable us to derive the characterisation of L\'evy measures on Besov spaces.

Having characterised the L\'evy measures on  Besov spaces, we  turn our attention to the regularisation question. Given a (non-Gaussian) cylindrical L\'evy process $L$ in $L^2(\R^d)$, we give sharp results for when $L$ becomes a classical L\'evy process  in a fixed Besov space $B$. Our technique is to study when the cylindrical L\'evy measure of $L$ may be extended to a $\sigma$-additive measure on the Borel $\sigma$-algebra in $B$ which is a L\'evy measure on $B$; in this case, we are then able to show the existence of a regularised version of $L$ in $B$; see Definition \ref{def:Induced} for the precise notion of the latter. 

The result on the existence of a regularised version in a Besov space $B$ of a cylindrical L\'evy process in $L^2(\R^d)$ is
closely related to the so-called 0-Radonifying property of the embedding $\iota$ of $L^2(\R^d)$ into $B$. Here, $\iota$ is called $0$-Radonifying if the image of each cylindrical random variable in $L^2(\R^d)$ under $\iota$ becomes a classical random variable in $B$. If this is the case, it is an easy conclusion, that each cylindrical L\'evy process in $L^2(\R^d)$ has a regularised version in $B$. The result on the regularisation of cylindrical L\'evy process and its application to specific examples enable us to exactly characterise those Besov spaces for which the embedding $\iota$ is $0$-Radonifying 
or $p$-Radonifying; the latter is a specialisation of $0$-Radonifying under some moment conditions.

Our definition of and the fundamental results we shall need regarding Besov spaces are presented in Section \ref{sec:Besov}.
In Section \ref{sec:LMandLPinBesov}, we characterise the L\'evy measures on Besov spaces.
% and to apply this result to give conditions for when a given L\'evy process in $\S^\ast$ has a version in a given weighted Besov space.
In Section \ref{sec:Regularisation} we 
%demonstrate the regularisation of every $L^2$-cylindrical L\'evy process as a $\S^\ast$-valued L\'evy process, and then proceed to 
give a general characterisation of when there exists a regularised version in a specific Besov space for a cylindrical L\'evy processes in $L^2(\R^d)$. We finish the presentation in this section by characterising Radonifying embeddings of $L^2$ into  Besov spaces.
Finally, in Section \ref{sec:Applications}, we study in depth two important examples of cylindrical L\'evy process: the canonical symmetric-$\alpha$-stable process, and the cylindrical L\'evy process representable as an infinite sum of independent one-dimensional L\'evy processes, which we call the \emph{hedgehog process}. In both cases we present a full characterisation of the parameter set where the cylindrical L\'evy process has and has not a regularised version in a specific Besov space.
%%%%%%%%%%%%%%%%%%%%%%%%
\subsubsection*{Notation}

We take $\N=\{1,2,\ldots\}$ and $\Zp=\N\cup\{0\}$ and $\Rp=\{x\in\R\colon x\ge 0\}$. All vector spaces are over $\mathbb{R}$.
We shall assume throughout the text that we are working in $\R^d$ for a fixed dimension $d$, and we fix a probability space $(\Omega,\A,P)$. 
Sequences are referred to by $(x_i)_{i\in I}$; stochastic processes are denoted $(f(i)\colon i\in I)$.

Given a normed space $(U,\norm{\cdot}_U)$, we use the notation $B_{U}:=\{f\in U\colon  \norm{f}_{U} \le 1\}$ for the closed unit ball in $U$. 
For a topological vector space $(T,\tau)$, we denote the Borel $\sigma$-algebra generated by the open subsets of $T$ by $\Borel(T)$ and we denote the continuous (topological) dual space by $T^*$, referred to henceforth simply as the dual. The dual pairing is denoted $\scapro{t}{t^\ast}_T$ for $t\in T, t^\ast\in T^\ast$. 
For topological vector spaces $T$ and $S$ we denote the continuous linear operators from $T$ to $S$ by $\mathcal{L}(T,S)$ and $\mathcal{L}(T):=\mathcal{L}(T,T)$.

Given a measure space $(S,\mathcal{A},\mu)$, the space of $\mu$-equivalence classes of 
measurable functions $f\colon S\to\R$ is denoted by $L^0(S, \mu)$, and of 
$p$-th integrable functions by $L^p(S, \mu)$ for $p>0$. For a Borel measure $\mu$ on $S$ we define the reflected measure $\mu^-$ by $\mu^-(A):=\mu(-A)$ for each $A\in\Borel(S)$.
The Lebesgue measure on $\Borel(\R^d)$ is denoted by $\leb$. For the case $L^p(\R^d,\leb)$ we shall just write $L^p(\R^d)$. 
We equip $L^0(S,\mu)$ with the topology of convergence in measure, and we equip the spaces $L^p(S, \mu)$ for $p>0$ with their standard metrics and (quasi-)norms, denoted $\norm{\cdot}_{L^p(S,\mu)}$ or, where there is no risk of confusion, $\norm{\cdot}_{L^p}$. For $p>1$ we define $p'=\tfrac{p}{p-1}$ to be the conjugate of $p$ with the usual modification for $p\in\{1,\infty\}$.

For $a\in\Zp\cup\{\infty\}$ and an open set $B\subseteq\R^d$ we denote by $C^a(B)$ the set of real-valued bounded uniformly continuous functions on $B$ with bounded uniformly continuous $a$-th derivative, where $C(B)=C^0(B)$ denotes the bounded uniformly continuous functions without reference to differentiability, and $a=\infty$ denotes the functions with bounded uniformly continuous derivatives of all orders. Furthermore, $C_c^a(B)$ denotes the subset of $C^a(B)$ with compact support within $B$.

We shall write $a\lesssim b$ to mean that there exists a positive constant $C$ such that $a\le Cb$. If the constant $C$ depends on the parameters $p_1,\ldots,p_n$, we shall also write $C=C(p_1,\ldots,p_n)$ and $\lesssim_{p_1,\ldots p_n}$. 
The expression $a\eqsim b$ is equivalent to $a\lesssim b\lesssim a$.

%%%%%%%%%%%%%%%%%%%%%%%%
\section{Weighted Besov Spaces}
\label{sec:Besov}
The weighted Besov spaces $\Besov{p}{s}{w}$ with $0<p\le\infty$ and $s,w\in\R$ have many equivalent definitions, see \cite{Triebel2006,Triebel2008} for a comprehensive treatment. We shall use a definition based upon wavelets.
As we are focused on separable reflexive Banach spaces in this work, we shall use the scale $1<p<\infty$. 
%Furthermore, we shall consider all our spaces of functions and distributions as subspaces of the space $\S^\ast(\R^d)$ of tempered distributions; this shall be the starting point of our definitions.

Let $\S(\R^d)$ denote the Schwartz space of rapidly decreasing functions on $\R^d$, that is
	\begin{align*}
		\S(\R^d)
		:=\big\lbrace
			f\in C^{\infty}(\R^d)\colon
			\norm{f}_{\S_r} < \infty \text{ for all } r\in\Zp
		\big\rbrace, 
	\end{align*}
where the seminorms $\norm{\cdot}_{\S_r}$, $r\in\Zp$, are defined by 
\begin{align}
\label{eq:SchwartzSeminorm}
	\norm{f}_{\S_r}
		:= \max_{\abs{s}\le r}\sup_{x\in\R^d} (1+\abs{x}^2)^r \abs{\partial^s f(x)},
\end{align}
with $s=(s_1,\ldots, s_d)\in\mathbb{Z}_+^d$ and $
\partial^s:=\partial^{\abs{s}}/(\partial x_1^{s_1}\cdots\partial x_d^{s_d})$. 
With the topology generated by the family of seminorms $(\norm{\cdot}_{\S_r})_{r\in\Zp}$, the space $\S(\R^d)$ is metrisable, and $f_n\to f$ in $\S(\R^d)$ means $\norm{f_n-f}_{\S_r}\to 0$ for each $r\in\Zp$. 
The dual space of $\S(\R^d)$ is the space $\S^\ast(\R^d)$ of tempered distributions, which we shall equip with the strong topology; that is the topology generated by the family of seminorms $\{\eta_B\}$, where for each bounded $B\subseteq \S(\R^d)$ we define $\eta_B(f):=\sup_{\phi\in B}\abs{\scapro{\phi}{f}}$ for $f\in\S^{\ast}(\R^d)$. With this topology, $\S(\R^d)$ is reflexive and separable and densely embedded in $\S^\ast(\R^d)$; see \cite[Th.\ V.14~Cor.~1]{ReedSimon1980}. 
The duality in $\S(\R^d)$ is denoted by $\Scapro{f}{g}$ for $f\in\S(\R^d)$ and $g\in\S^\ast(\R^d)$. 

The spaces $L^p(\R^d)$ for $1\le p\le\infty$ may be interpreted as subspaces of $\S^\ast(\R^d)$ in the following manner: for each $g\in L^p(\R^d)$ the map
\begin{align*}
	f\mapsto \int_{\R^d}f(x)g(x)\,\d x,
	\qquad f\in\S(\R^d),
\end{align*}
i.e.\ the duality in $L^{p'}(\R^d)$, is finite for all $f\in\S(\R^d)$ by H\"older's inequality. 
In this manner we identify each $g\in L^p(\R^d)$ with a functional in $\S^\ast(\R^d)$, which we shall also refer to as $g$.
Furthermore we extend this identification to the dual $B^\ast$ of
any Banach space $B$ in which $\S(\R^d)$ is densely embedded, and
we identify $\scapro{f}{g}_B\equiv\Scapro{f}{g}$ for each $f\in\S(\R^d)$ and $g\in B^\ast$. 
Finally, we shall extend the definition of $\Scapro{f}{g}$  as follows: if there exists a Banach space $B$ in which $\S(\R^d)$ is dense and $f\in B$ and $g\in B^\ast$ then we define $\Scapro{f}{g}:=\scapro{f}{g}_B$, otherwise we take $\Scapro{f}{g}:=\infty$.
However, this interpretation means that we do \emph{not} identify Hilbert spaces with their duals, except in the case of $L^2(\R^d)$.  

We shall define the weighted Besov spaces $\Besov{p}{s}{w}$ for $p>1$ and $s,w\in\R$ in terms of wavelet bases of $L^2(\R^d)$, for which we follow  \cite[Sec.~1.2.3]{Triebel2008}.
We define subsets $G^j\subseteq \{0,1\}^d, j\in\Zp$ by
\begin{align*}
	G^j:=
	\begin{cases}
		\{0,1\}^d,&\text{if }j=0,\\
		\{G=(G_1,\ldots,G_d)\colon G_i=1 \text{ for at least one }i\},&\text{if }j\ge 1.
	\end{cases}
\end{align*}
Suppose we are given $\big(\Psi^G_0\big)_{G\in G^0}\subseteq C_c(\R^d)$ which form an orthonormal set in $L^2(\R^d)$, which we shall call the \emph{parent wavelets}. Then, for each $j\in\Zp, G\in G^j$ and $m\in\Z^d$, we define 
\begin{align*}
	\Psi^{j,G}_m(x):=
	2^{jd/2} \Psi^G_m(2^{j}x):=
	2^{jd/2} \Psi^G_0(2^{j}x-m),
	\qquad x\in\R^d.
\end{align*}
It is known that for any $r\in\N$, there exist such parent wavelets $\big(\Psi^G_0\big)_{G\in G^0}\subseteq C^r_c(\R^d)$ such that $\Psi:=\{\Psi^{j,G}_m\colon j\in\Zp, G\in G^j, m\in\Z^d\}$ forms an orthonormal basis in $L^2(\R^d)$; see \cite[Th.~1.61]{Triebel2006}; one example is the Daubechies wavelets \cite{Daubechies1992}.
In this case,  $\Psi$ is called a \emph{wavelet basis of $L^2(\R^d)$}.
Henceforth, we shall refer to the wavelet index set
\begin{align}
\label{eq:WaveletIndex}
	\mathbb{W}^d:=\left\lbrace (j,G,m)\colon  j\in\Zp, G\in G^j, m\in\Z^d \right\rbrace.
\end{align}
Since $\mathbb{W}^d$ is countable, we may consider in the usual manner the space $\ell^p(\mathbb{W}^d)$ of $p$-summable sequences with index set $\mathbb{W}^d$.
For the purposes of defining the weighted Besov space $\Besov{p}{s}{w}$, we shall require a minimum smoothness of the wavelet basis depending on the dimension $d$ and the parameters $p$ and $s$.

\begin{definition}
\label{def:AdmissibleBasis}
	Let $p>1$ and $s$, $w\in\R$. 
	A wavelet basis $\Psi=\{\Psi^{j,G}_m\colon (j,G,m)\in\mathbb{W}^d\}$ of $L^2(\R^d)$ is called an \emph{admissible basis of $\Besov{p}{s}{w}$} if $\Psi\subseteq C_c^r(\R^d)$ for some integer $r>\abs{s}$.
\end{definition}

We begin with defining the \emph{weighted Besov sequence space} $b^{p}_{s,w}$ for $p>1$ and $s$, $w\in\R$, for which we 
 introduce the weight constants:
\begin{align}
\label{eq:Weights}
	\omega_m^j=\omega_m^j(p,s,w):=2^{j(s-\frac{d}{p}+\frac{d}{2})}(1+2^{-2j}\abs{m}^2)^{\frac{w}{2}},
\end{align}
for each $m\in\Z^d$ and $j\in\Zp$.
We define $b^{p}_{s,w}$
as the vector space of real-valued sequences $(\lambda^{j,G}_m)_{(j,G,m)\in\mathbb{W}^d}$ 
such that
\begin{align*}
	\norm{\lambda}_{b^{p}_{s,w}}
	:= \Bigg(\sum_{j\in\Zp}\sum_{G\in G^j}\sum_{m\in\Z^d}\abs{2^{-\frac{jd}{2}}\omega_m^j\lambda^{j,G}_m}^p\Bigg)^{1/p}
	<\infty.
\end{align*}
The space $(b^{p}_{s,w},\norm{\cdot}_{b^{p}_{s,w}})$ forms a Banach space if $p>1$ and is a Hilbert space for $p=2$.

 Let $\Psi$ be an admissible basis of $\Besov{p}{s}{w}$ for some  $p>1$ and $s$, $w\in\R$. The \emph{weighted Besov space} $\Besov{p}{s}{w}$ is defined to be
\begin{align}\label{de.Besov}
	\Besov{p}{s}{w}:=\left\lbrace
		f\in\mathcal{S}^{\ast}(\R^d)\colon 
		f=\sum_{j\in\Zp}\sum_{G\in G^j}\sum_{m\in\Z^d}\lambda^{j,G}_m 2^{-\frac{jd}{2}}\Psi^{j,G}_m,\,
		\lambda\in b^{p}_{s,w}
	\right\rbrace,
\end{align}
where the sum is unconditionally convergent in $\mathcal{S}^{\ast}(\R^d)$. 
When this holds, the associated sequence $\lambda$ is unique and we 
have $\lambda^{j,G}_m=2^{\frac{jd}{2}}\Scapro{\Psi^{j,G}_m}{f}$. 
A consequence of $\Psi$ being an admissible basis of $\Besov{p}{s}{w}$ is that the wavelets are of sufficient smoothness to guarantee that they are in $\big(\Besov{p}{s}{w}\big)^\ast$, and so the dual pairing makes sense.
As the sums over $j,G$ and $m$ are unconditional in the definitions of both the weighted Besov spaces and the weighted Besov sequence spaces, we will henceforth use the simpler notation $\sum_{j,G,m}$ to mean $\sum_{j\in\Zp}\sum_{G\in G^j}\sum_{m\in\Z^d}$.
We may norm $\Besov{p}{s}{w}$ by taking $\norm{f}_{\besov{p}{s}{w}}:=\norm{\lambda}_{b^{p}_{s,w}}$, giving
\begin{align}
\label{eq:NormB}
	\norm{f}_{\besov{p}{s}{w}} = 
	\Bigg(\sum_{j,G,m} (\omega_m^j)^p\abs{\Scapro{\Psi^{j,G}_m}{f}}^p
	\Bigg)^{1/p}.
\end{align}
It follows that $\Besov{p}{s}{w}$ is a Banach space for $p>1$ and a Hilbert space for $p=2$.
We immediately see from this definition that $\Besov{2}{0}{0}=L^2(\R^d)$, consistently with the relation described in the introduction.

\begin{remark}
	The Besov space scale typically has an auxiliary parameter $q$.
	However, we shall not be using the $q$ parameter in this work; 
this is due to the fact that the embedding theorems available in weighted Besov spaces (see Proposition 3 in \cite{Fageot2016}) show the continuous embedding of $\Besov{p,q}{s}{w}$ into $\Besov{p,p}{s-\epsilon}{w}$ for any $q\in\R$ and $\epsilon>0$. The results presented in this work are generally expressed as strict inequalities on the Besov space parameters, and as such are unaffected by the arbitrarily small change in the $s$ parameter needed to incorporate any $q$ parameter. 
\end{remark}

The dual spaces for the classical unweighted Besov spaces are well-known: $\big(\Besov{p}{s}{0}\big)^\ast=\Besov{p'}{-s}{0}$ for $p>1$ and $s\in\R$; see e.g. \cite[p.~179]{Triebel1983}. We may easily generalise to the weighted case to state the following result:
\begin{theorem}
\label{thm:BesovDuality}
	Let $p>1$ and $s$, $w\in\R$. The dual space $\big(\Besov{p}{s}{w}\big)^\ast$ may be identified with $\Besov{p'}{-s}{-w}$
	with $p':=\tfrac{p}{p-1}$, and the duality given by
	\begin{align}
	\label{eq:BesovDuality}
		\scapro{f}{g}_{\besov{p}{s}{w}}=\Scapro{f}{g}
		= \sum_{j,G,m}\Scapro{\Psi_m^{j,G}}{f}\Scapro{\Psi_m^{j,G}}{g},
	\end{align}
	where $\Psi$ is any admissible basis for $\Besov{p}{s}{w}$ (and thus is an admissible basis for $\Besov{p'}{-s}{-w}$).
\end{theorem}

Weighted Besov spaces form various scales according to the parameters $p, s$ and $w$; we present a general result for their continuous embeddings. The positive results, for when a certain Besov space is continuously embedded in another given Besov space, are well-known; however, we are unaware of any converse results so we develop such converses here.

\begin{proposition}
\label{prop:BesovEmbeddings}
	Let $s_0,s_1,w_0,w_1\in\R$ and $p_0,p_1>1$.
	\begin{enumerate}
	\item[{\rm (1)}] Suppose $p_0> p_1$. Then $\Besov{p_0}{s_0}{w_0} \hookrightarrow \Besov{p_1}{s_1}{w_1}$ continuously if and only if
		\begin{align*}
			s_0>s_1 
			\qquad\text{and}\qquad 
			w_0-w_1 > d\left(\frac{1}{p_1}-\frac{1}{p_0}\right).
		\end{align*}
	\item[{\rm (2)}]  Suppose $p_0\le p_1$. Then $\Besov{p_0}{s_0}{w_0} \hookrightarrow \Besov{p_1}{s_1}{w_1}$ continuously if and only if
		\begin{align*}
			s_0-s_1\ge d\left(\frac{1}{p_0}-\frac{1}{p_1}\right)
			\qquad\text{and}\qquad 
			w_0\ge w_1.
		\end{align*}
	\end{enumerate}
\end{proposition}

In order to prove this Proposition, we shall first prove an intermediary result about the isomorphism between weighted Besov sequence spaces as defined above and $\ell^p$ spaces. 
\begin{lemma}
\label{lem:Isomorphism}
For each $p>1$ and $s,w\in\R$ the operator $\Upsilon_{s,w}^p\colon b^{p}_{s,w}\to\ell^p(\mathbb{W}^d)$ defined by
\begin{align}
\label{eq:Isomorphism}
	\big(\Upsilon_{s,w}^p\lambda\big)_m^{j,G}
	:= 2^{-\frac{jd}{2}}\omega_m^j\lambda^{j,G}_m
\end{align}
forms an isometric isomorphism, where the constants $\omega_m^j=\omega_m^j(p,s,w)$ are  defined in \eqref{eq:Weights}.
\end{lemma}
\begin{proof}
Since 	$\norm{\lambda}_{b^p_{s,w}}=\norm{\Upsilon_{s,w}^p\lambda}_{\ell^p(\mathbb{W}^d)}$ by the very definition, the proof 
is elementary. 
\end{proof}

\begin{proof}[Proof of Proposition \ref{prop:BesovEmbeddings}]
	The continuous embedding in both cases is given by Proposition 3 in \cite{Fageot2016}.
	To prove non-inclusion we first note that for any $q>0$, if $y=(y_m^{j,G})_{(j,G,m)\in\mathbb{W}^d}$ is such that $y\notin\ell^q(\mathbb{W}^d)^\ast$ then there exists $x\in\ell^q(\mathbb{W}^d)$ with $\sum_{j,G,m} x_m^{j,G}y_m^{j,G}=\infty$.
	
	(1): suppose that $p_0> p_1$ and $w_1-w_0\ge -d\big(\tfrac{1}{p_1}-\tfrac{1}{p_0}\big)$. Taking $\alpha=\tfrac{p_0}{p_0-p_1}$ and thus $\alpha'=\tfrac{p_0}{p_1}$, we have $\alpha p_1(w_1-w_0)\ge -d$. Since for any $j\ge 0$ we have
	\begin{align*}
		\sum_{m\in\Z^d}\big(1+2^{-2j}\abs{m}^2\big)^{\frac{\alpha p_1(w_1-w_0)}{2}}=\infty, 
	\end{align*}
 according to the proof of \cite[Th.~3]{Fageot2016}, we conclude
	\begin{align*}
		\sum_{j,G,m}2^{j\alpha p_1(s_1-s_0-\frac{d}{p_1}+\frac{d}{p_0})}\big(1+2^{-2j}\abs{m}^2\big)^{\frac{\alpha p_1(w_1-w_0)}{2}}=\infty.
	\end{align*}
   Since the last expression means for the weights defined in\eqref{eq:Weights} that 
	\begin{align*}
		\left(\left( \frac{\omega_m^j(p_1,s_1,w_1)}{\omega_m^j(p_0,s_0,w_0)} \right)^{p_1}\right)_{(j,G,m)\in\mathbb{W}^d}\notin\ell^\alpha(\mathbb{W}^d),
	\end{align*}
  it follows that there exists a non-negative $y\in\ell^{\frac{p_0}{p_1}}(\mathbb{W}^d)$ satisfying
	\begin{align*}
		\sum_{j,G,m} \left( \frac{\omega_m^j(p_1,s_1,w_1)}{\omega_m^j(p_0,s_0,w_0)} \right)^{p_1}y_m^{j,G}=\infty.
	\end{align*}
	The isometry between $\ell^{p_0}(\mathbb{W}^d)$ and $b^{p_0}_{s_0,w_0}$ established in Lemma \ref{lem:Isomorphism} guarantees that 
	$
		\lambda_m^{j,G} = 2^{\frac{jd}{2}}\big(\omega_m^j(p_0,s_0,w_0)\big)^{-1}(y_m^{j,G})^{\frac{1}{p_1}}
	$
    defines a sequence $\lambda:=(\lambda_m^{j,G})$ in $b^{p_0}_{s_0,w_0}$. Since 
	\begin{align*}
		\norm{\lambda}_{b^{p_1}_{s_1,w_1}}^{p_1}
		=\sum_{j,G,m} \big(2^{-\frac{jd}{2}}\omega_m^j(p_1,s_1,w_1)\lambda_m^{j,G}\big)^{p_1}
		= \sum_{j,G,m} \left( \frac{\omega_m^j(p_1,s_1,w_1)}{\omega_m^j(p_0,s_0,w_0)} \right)^{p_1}y_m^{j,G}=\infty,
	\end{align*}
    it follows $\lambda\notin b^{p_1}_{s_1,w_1}$, which shows that $b^{p_0}_{s_0,w_0}\nsubseteq b^{p_1}_{s_1,w_1}$ which in turn shows $\Besov{p_0}{s_0}{w_0}\nsubseteq \Besov{p_1}{s_1}{w_1}$ by the isometry between the weighted Besov spaces and the weighted Besov sequence spaces; see \cite[Th.~6.15]{Triebel2006}.
	
	Now suppose $p_0> p_1$  and $s_0\le s_1$, and define for $(j,G,m)\in\mathbb{W}^d$
	\begin{align*}
		x_m^{j,G}:=\left( \frac{\omega_m^j(p_1,s_1,w_1)}{\omega_m^j(p_0,s_0,w_0)} \right)^{p_1}.
	\end{align*} 
	We can assume 
		$
		S_j:=\sum_{m\in\Z^d}\big(1+2^{-2j}\abs{m}^2\big)^{\frac{\alpha p_1(w_1-w_0)}{2}}<\infty
	$ for each $j\in\Zp$, as otherwise there would be nothing to prove. Since $S_j$ is asymptotically $\O(2^{jd})$ as $j\to\infty$ according to \cite[Th.~3]{Fageot2016}, we obtain
	\begin{align*}
		\sum_{j,G,m}\abs{x_m^{j,G}}^\alpha
		= \sum_{j,G}2^{\alpha jp_1(s_1-s_0-\frac{d}{p_1}+\frac{d}{p_0})}S_j=\infty,
	\end{align*}
	as $\alpha p_1\big(s_1-s_0-\tfrac{d}{p_1}+\tfrac{d}{p_0}\big)\ge -d$. It follows $\big(x_m^{j,G}\big)_{(j,G,m)\in\mathbb{W}^d}\notin \ell^\alpha(\mathbb{W}^d)$, and the proof of non-inclusion proceeds as above.
	
	(2): let $p_0\le p_1$ and $s_1-s_0>d(\tfrac{1}{p_1}-\tfrac{1}{p_0})$,  so we have $2^{jp_1(s_1-s_0-\frac{d}{p_1}+\frac{d}{p_0})}$ is unbounded as $j\to\infty$. For $(j,G,m)\in\mathbb{W}^d$ we define
	\begin{align*}
		x_m^{j,G}:=\left( \frac{\omega_m^j(p_1,s_1,w_1)}{\omega_m^j(p_0,s_0,w_0)} \right)^{p_1}.
	\end{align*}
	Since  $\big(x_m^{j,G}\big)_{(j,G,m)\in\mathbb{W}^d}\notin \ell^\infty(\mathbb{W}^d)$ there exists $y\in\ell^{\frac{p_0}{p_1}}(\mathbb{W}^d)$ such that
	\begin{align*}
		\sum_{j,G,m} \left( \frac{\omega_m^j(p_1,s_1,w_1)}{\omega_m^j(p_0,s_0,w_0)} \right)^{p_1}y_m^{j,G}=\infty,
	\end{align*}
	where we recall the dual of $\ell^p$ for $p\le 1$ is $\ell^\infty$.
	The proof of the non-inclusion follows as above.
	
	In the remaining case $p_0> p_1$  and $w_0<w_1$, we again obtain $\big(x_m^{j,G}\big)_{(j,G,m)\in\mathbb{W}^d}\notin \ell^\infty(\mathbb{W}^d)$, and the non-inclusion result follows.
\end{proof}

We may now identify the set of weighted Besov spaces containing $L^2(\R^d)$.
\begin{proposition}
	\label{prop:LqEmbeddings}
	Let $p>1$ and define
	\begin{align}
		\label{eq:EmbeddingRegion}
		E_{p}:=\begin{cases}
			(-\infty,0) \times (-\infty,-\tfrac{d}{p}+\tfrac{d}{2}), & \text{if }p\in (1,2),\\
			(-\infty,0] \times (-\infty,0], & \text{if }p=2,\\
			(-\infty,-\tfrac{d}{2}+\tfrac{d}{p}) \times (-\infty,0], & \text{if }p\in (2,\infty).
		\end{cases}
	\end{align}
	Then $L^2(\R^d) \subseteq \Besov{p}{s}{w}$ if and only if $(s,w)\in E_p$; in this case, 
    the embedding is continuous.
\end{proposition}

\begin{proof}
	We recall that $L^2(\R^d)=\Besov{2}{0}{0}$. The result then follows by applying Proposition \ref{prop:BesovEmbeddings} for the case $p_0=2$, $p_1=p$, $s_0=0$, $s_1=s$, $w_0=0$ and $w_1=w$.
\end{proof}

%%%%%%%%%%%%%%%%%%%%%%%%%%%%%%%%%%%%%%%%
\section{L\'evy measures and L\'evy processes in Besov spaces}
\label{sec:LMandLPinBesov}

Let $U$ be a separable Banach space with dual $U^\ast$.
We define L\'evy processes in $U$ in the usual manner, that is, a $U$-valued process $L=(L(t)\colon t\ge 0)$ such that $L(0)=0$, $L$ has independent and stationary increments, and the map from $t$ to the law of $L(t)$ is continuous at 0 in the topology of weak convergence of probability measures. 

 For an arbitrary finite measure $\mu$ on $\Borel(U)$ define the exponential measure $e(\mu)$ by
\begin{align}
\label{eq:ExpMeasure}
	e(\mu):=e^{-\mu(U)} \sum_{m=0}^\infty \frac{1}{m!}\mu^{\ast m}.
\end{align}
The  exponential measure $e(\mu)$ is a compound Poisson distribution  with characteristic function 
\begin{align*}
\phi_\mu\colon U^\ast\to\C, \qquad 
	\phi_\mu(u^\ast)=\exp\left(\int_U\left(e^{i\scapro{u}{u^\ast}_U}-1\right)\,\mu(\d u)\right).
\end{align*}
Whereas in the finite dimensional case, and in Hilbert spaces, the integrability of $\abs{\cdot}^2\wedge 1$ characterises  L\'evy measures, there are no equivalent conditions known in arbitrary Banach spaces for when the same function $\phi_\mu$ but for a $\sigma$-finite measure $\mu$ still forms the characteristic function of a probability measure. Thus, L\'evy measures are defined implicitly in the following way, see \cite{Linde1986}: 
\begin{definition}
\label{def:LevyMeasure}
	A $\sigma$-finite measure $\mu$ on $\Borel(U)$ is a \emph{L\'evy measure} if $\mu(\{0\})=0$ and 
	\begin{align*}
		\phi_\mu\colon U^\ast\to\C, \qquad 
		\phi_\mu(u^\ast)=\exp\left(\int_U\left(e^{i\scapro{u}{u^\ast}_U}-1-i\scapro{u}{u^\ast}_U\1_{B_U}(u)\right)\,\mu(\d u)\right)
	\end{align*}
is the characteristic function of a probability measure on $\Borel(U)$, which we shall call $e_S(\mu)$.
\end{definition}
Theorem 5.4.8 in  \cite{Linde1986} guarantees that a $\sigma$-finite measure $\mu$ on $\Borel(U)$ is a L\'evy measure if and only if its symmetrisation $\mu+\mu^-$ is a L\'evy measure.  

L\'evy measures in sequence spaces $\ell^p$ are characterised in Yurinskii \cite{Yurinskii1974}. Using the wavelet definition of weighted Besov spaces enables us to derive a corresponding result for the L\'evy measures on these spaces. 
\begin{theorem}
\label{thm:LevyMeasureBesov}
	A  $\sigma$-finite measure $\mu$ on $\Borel(\Besov{p}{s}{w})$ with $\mu(\{0\})=0$ is a L\'evy measure on $\Besov{p}{s}{w}$ for some $p\in (1,\infty)$ and $s$, $w\in\R$ if and only if 
	\begin{enumerate}
	\item[{\rm (1)}] for $p\ge 2$, 
	\begin{align}
	\label{eq:BesovMeas3}
	&\int_{\besov{p}{s}{w}} \left(\norm{f}_{\besov{p}{s}{w}}^p \wedge 1\right)\, \mu(\d f)<\infty,\\
	\label{eq:BesovMeas2}
	&\sum_{j,G,m} (\omega_m^j)^p \left(\int_{\norm{f}_{\besov{p}{s}{w}}\le 1} \Scapro{\Psi_m^{j,G}}{f}^2\,\mu({\rm d}f)\right)^{p/2}<\infty;
	\end{align}
	\item[{\rm (2)}] and for $p\in (1,2)$,
	\begin{align}
	\label{eq:BesovMeas1}
	&\int_{\besov{p}{s}{w}} \left(\norm{f}_{\besov{p}{s}{w}}^2 \wedge 1\right)\, \mu(\d f)<\infty,\\
	\label{eq:BesovMeas5}
	& \sum_{j,G,m} (\omega_m^j)^p \int_0^\infty \left(1-e^{\int_{\norm{f}_{\besov{p}{s}{w}}\le 1}\big(\cos\tau\Scapro{\Psi_m^{j,G}}{f}-1\big)\,\mu({\rm d}f)}\right)\,\frac{{\rm d}\tau}{\tau^{1+p}}<\infty.
	\end{align}
	\end{enumerate}
	In the expressions above, $\omega_m^j=\omega_m^j(p,s,w)$ are the weight constants defined in \eqref{eq:Weights}, 
	and $\{\Psi^{j,G}_m\colon j\in\Zp, G\in G^j, m\in\Z^d\}$ is an admissible basis of $\Besov{p}{s}{w}$.
\end{theorem}

\begin{proof}
	We begin by showing sufficiency. Because of \cite[Th.~5.4.8]{Linde1986}, we can assume that $\mu$ is symmetric.
    Given $0\le\alpha\le\beta\le\infty$ we define
    \begin{align*}
    	\mu_{\alpha,\beta}(A):=\mu\big(A\cap\{f\in \Besov{p}{s}{w}\colon \alpha\le\norm{f}_{\besov{p}{s}{w}}<\beta\}\big),
    	\qquad A\in\Borel(\Besov{p}{s}{w}).
    \end{align*}	
	 	Let $\epsilon\in[0,1)$  and note that Conditions \eqref{eq:BesovMeas3} and \eqref{eq:BesovMeas1} each imply that $\mu_{\epsilon,1}(\Besov{p}{s}{w})$ is finite. Thus, the exponential measure $e(\mu_{\epsilon,1})$ defined in \eqref{eq:ExpMeasure} coincides with $e_S(\mu_{\epsilon,1})$ and we obtain 
	\begin{align}\label{eq.aux3}
		\int_{\besov{p}{s}{w}}\norm{f}_{\besov{p}{s}{w}}^p\,e_S(\mu_{\epsilon,1})({\rm d}f)
		&= \sum_{j,G,m}(\omega_m^j)^p\int_{\besov{p}{s}{w}}\abs{\Scapro{\Psi_m^{j,G}}{f}}^p\,e_S(\mu_{\epsilon,1})({\rm d}f)\notag\\
		&= \sum_{j,G,m}(\omega_m^j)^p\int_{\R}\abs{\beta}^p\,\big(e_S(\mu_{\epsilon,1})\circ\Scapro{\Psi_m^{j,G}}{\cdot}^{-1}\big)({\rm d}\beta)\notag\\
		&= \sum_{j,G,m}(\omega_m^j)^p E\abs{\xi_m^{j,G}}^p,
	\end{align}
	where $\xi_m^{j,G}$ is a random variable with distribution $e_S(\mu_{\epsilon,1})\circ\Scapro{\Psi_m^{j,G}}{\cdot}^{-1}=e_S(\mu_{\epsilon,1}\circ\Scapro{\Psi_m^{j,G}}{\cdot}^{-1})$
	for each $j\in\Zp, G\in G^j$ and $m\in\Z^d$.
	
	For $p\ge 2$, Theorem 1.1 in \cite{Dirksen2014} guarantees 
	\begin{align*}%\label{eq.aux4}
		E\abs{\xi_m^{j,G}}^p \eqsim_p 
		\int_{\R}\abs{\beta}^p\,(\mu_{\epsilon,1}\circ\Scapro{\Psi_m^{j,G}}{\cdot}^{-1})({\rm d}\beta)
			+ \left(\int_{\R}\abs{\beta}^2\,(\mu_{\epsilon,1}\circ\Scapro{\Psi_m^{j,G}}{\cdot}^{-1})({\rm d}\beta)\right)^{p/2}.
	\end{align*}
   It follows that, for $p\ge 2$,
	\begin{align}
	\label{eq:pBounds}
		&\int_{\besov{p}{s}{w}}\norm{f}_{\besov{p}{s}{w}}^p\,e_S(\mu_{\epsilon,1})({\rm d}f)\notag\\
		&\eqsim_p \sum_{j,G,m}(\omega_m^j)^p \left(
			\int_{\R}\abs{\beta}^p\,(\mu_{\epsilon,1}\circ\Scapro{\Psi_m^{j,G}}{\cdot}^{-1})({\rm d}\beta)
			+ \left(\int_{\R}\abs{\beta}^2\,(\mu_{\epsilon,1}\circ\Scapro{\Psi_m^{j,G}}{\cdot}^{-1})({\rm d}\beta)\right)^{p/2}
			\right)\notag\\
		&= \sum_{j,G,m}(\omega_m^j)^p \left(
			\int_{\besov{p}{s}{w}}\abs{\Scapro{\Psi_m^{j,G}}{f}}^p\,\mu_{\epsilon,1}({\rm d}f)
			+ \left(\int_{\besov{p}{s}{w}}\abs{\Scapro{\Psi_m^{j,G}}{f}}^2\,\mu_{\epsilon,1}({\rm d}f)\right)^{p/2}
			\right)\notag\\
		&= \int_{\besov{p}{s}{w}}\norm{f}_{\besov{p}{s}{w}}^p\,\mu_{\epsilon,1}({\rm d}f) + \sum_{j,G,m}(\omega_m^j)^p\left(\int_{\besov{p}{s}{w}}\abs{\Scapro{\Psi_m^{j,G}}{f}}^2\,\mu_{\epsilon,1}({\rm d}f)\right)^{p/2}.
	\end{align}
	Conditions \eqref{eq:BesovMeas3} and \eqref{eq:BesovMeas2} imply
	\begin{align}
	\label{eq:YurinsCond1}
		\sup_{\epsilon\in(0,1)}\int_{\besov{p}{s}{w}}\norm{f}_{\besov{p}{s}{w}}^p\,e(\mu_{\epsilon,1})({\rm d}f)<\infty.
	\end{align}
   By applying H\"older's inequality twice,  Theorem \ref{thm:BesovDuality} shows for $g\in \big(\Besov{p}{s}{w}\big)^\ast$ that 
	\begin{align*}
			&\int_{\norm{f}_{\besov{p}{s}{w}}\le 1}\scapro{f}{g}_{\besov{p}{s}{w}}^2 \,\mu({\rm d}f)\notag\\
		&= \sum_{j,G,m} \sum_{k,H,n} \Scapro{\Psi_m^{j,G}}{g}\Scapro{\Psi_n^{k,H}}{g} \int_{\norm{f}_{\besov{p}{s}{w}}\le 1} \Scapro{\Psi_m^{j,G}}{f}\Scapro{\Psi_n^{k,H}}{f}\,\mu({\rm d}f)\notag\\
		&\le \left(\sum_{j,G,m} \Scapro{\Psi_m^{j,G}}{g} \left(\int_{\norm{f}_{\besov{p}{s}{w}}\le 1} \Scapro{\Psi_m^{j,G}}{f}^2\,\mu({\rm d}f)\right)^{1/2}\right)^2\notag\\
		&= \left(\sum_{j,G,m} (\omega_m^j)^{-1}\Scapro{\Psi_m^{j,G}}{g} \left((\omega_m^j)^2\int_{\norm{f}_{\besov{p}{s}{w}}\le 1} \Scapro{\Psi_m^{j,G}}{f}^2\,\mu({\rm d}f)\right)^{1/2}\right)^2\notag\\
		&\le \left(\sum_{j,G,m} (\omega_m^j)^{-p'}\abs{\Scapro{\Psi_m^{j,G}}{g}}^{p'}\right)^{2/p'} \left(\sum_{j,G,m} (\omega_m^j)^p\left(\int_{\norm{f}_{\besov{p}{s}{w}}\le 1} \Scapro{\Psi_m^{j,G}}{f}^2\,\mu({\rm d}f)\right)^{p/2}\right)^{2/p}.
	\end{align*}
	Since $\sum_{j,G,m} (\omega_m^j)^{-p'}\abs{\Scapro{\Psi_m^{j,G}}{g}}^{p'}=\norm{g}_{\besov{p'}{-s}{-w}}^{p'}$, Condition  \eqref{eq:BesovMeas2} guarantees 
	\begin{align*}
		\int_{\norm{f}_{\besov{p}{s}{w}}\le 1}\scapro{f}{g}_{\besov{p}{s}{w}}^2 \,\mu({\rm d}f)
		\lesssim\norm{g}_{\besov{p'}{-s}{-w}}^2.
	\end{align*}
	Together with \eqref{eq:YurinsCond1} we see that the conditions of Theorem 1.B in \cite{Yurinskii1974} are satisfied and hence $\mu$ is a L\'evy measure on $\Besov{p}{s}{w}$.
	
	Next we consider the case $p\in(1,2)$; in this case, we use the following relation
	for a real-valued symmetric random variable $X$ with characteristic function $\phi_X$ and $q\in(0,2)$:
	\begin{align}
		\label{eq:MomentFormula}
		E\abs{X}^q
		= c_q\int_0^\infty \frac{1-{\rm Re}(\phi_X(\tau))}{\tau^{q+1}}\,{\rm d}\tau,
	\end{align}
	where $c_q$ is a constant depending only on $q$; see for example \cite[Th.~11.4.3]{Kawata1972}.
 
   Applying Equality \eqref{eq:MomentFormula} to a random variable with distribution $e(\mu_{\epsilon,1})\circ\Scapro{\Psi_m^{j,G}}{\cdot}^{-1}$ shows, for each $\epsilon\in(0,1)$,
	\begin{align*}
		\int_{\besov{p}{s}{w}}&\norm{f}_{\besov{p}{s}{w}}^p\, e(\mu_{\epsilon,1})({\rm d}f)\\
		&=c_p \sum_{j,G,m}(\omega_m^j)^p
			\int_0^\infty \left(1-e^{\int_{\besov{p}{s}{w}}\big(\cos\tau \Scapro{\Psi_m^{j,G}}{f}-1\big)\,\mu_{\epsilon,1}({\rm d}f)}\right)\,\frac{{\rm d}\tau}{\tau^{1+p}}
			\\
		&\le c_p \sum_{j,G,m}(\omega_m^j)^p
			\int_0^\infty \left(1-e^{\int_{\norm{f}_{\besov{p}{s}{w}\le 1}}\big(\cos\tau \Scapro{\Psi_m^{j,G}}{f}-1\big)\,\mu({\rm d}f)}\right)\,\frac{{\rm d}\tau}{\tau^{1+p}}.
	\end{align*}
	 Condition \eqref{eq:BesovMeas5} guarantees
	\begin{align}\label{eq.aux1}
		\sup_{\epsilon\in(0,1)}\int_{\besov{p}{s}{w}}\norm{f}_{\besov{p}{s}{w}}^p\,e(\mu_{\epsilon,1})({\rm d}f)<\infty.
	\end{align}
    Condition \eqref{eq:BesovMeas1} implies for $g\in \Besov{p'}{-s}{-w}$ that
	\begin{align*}
		\int_{\norm{f}_{\besov{p}{s}{w}}\le 1}\scapro{f}{g}_{\besov{p}{s}{w}}^2\,\mu({\rm d}f)
		&\le \int_{\norm{f}_{\besov{p}{s}{w}}\le 1}\norm{f}_{\besov{p}{s}{w}}^2\norm{g}_{\besov{p'}{-s}{-w}}^2\,\mu({\rm d}f)\lesssim\norm{g}_{\besov{p'}{-s}{-w}}^2. 
	\end{align*}
 Together with \eqref{eq.aux1}, it follows that the conditions of Theorem 1.B in \cite{Yurinskii1974} are satisfied and hence $\mu$ is a L\'evy measure on $\Besov{p}{s}{w}$.
	
	To establish necessity of the conditions suppose that $\mu$ is a L\'evy measure on $\Besov{p}{s}{w}$; as before we can assume $\mu$ is symmetric. Proposition 5.4.1 in \cite{Linde1986} implies that we have $\mu\big(B_{\besov{p}{s}{w}}^c\big)<\infty$.
	Since $\mu_{0,1}$ is a symmetric L\'evy measure on $\Besov{p}{s}{w}$, the probability measure $e_S(\mu_{0,1})$ exists according to \cite[Cor.~5.4.4]{Linde1986}. Corollary~3.3 in \cite{Acosta1980}  shows, for all $q>0$, that 
	\begin{align}\label{eq.Acosta-poly}
		\int_{\besov{p}{s}{w}}\norm{f}_{\besov{p}{s}{w}}^q\,e_S(\mu_{0,1})({\rm d}f)<\infty.
	\end{align}
	We first consider the case $p\ge 2$. Since $e_S(\mu_{0,1})$ exists we obtain as in  \eqref{eq:pBounds} that
	\begin{align*}
		&\int_{\besov{p}{s}{w}}\norm{f}_{\besov{p}{s}{w}}^p\,\mu_{0,1}({\rm d}f)+\!\sum_{j,G,m}(\omega_m^j)^p \bigg(\int_{\besov{p}{s}{w}} \Scapro{\Psi_m^{j,G}}{f}^2\,\mu_{0,1}({\rm d}f)\bigg)^{p/2}
		\!\!\eqsim_p \int_{\besov{p}{s}{w}}\norm{f}_{\besov{p}{s}{w}}^p\,e_S(\mu_{0,1})({\rm d}f),
		%<\infty,
	\end{align*}
	which verifies the necessity of Conditions \eqref{eq:BesovMeas3} and \eqref{eq:BesovMeas2} because of \eqref{eq.Acosta-poly}.
	
	For the case $p\in(1,2)$, we apply Equality \eqref{eq:MomentFormula} to a random variable with distribution $e_S(\mu_{0,1})\circ\Scapro{\Psi_m^{j,G}}{\cdot}^{-1}$ to obtain
	\begin{align*}
		\sum_{j,G,m}(\omega_m^j)^p \int_0^\infty \left(1-e^{\int_{\norm{f}_{\besov{p}{s}{w}}\le 1}\big(\cos\tau \Scapro{\Psi_m^{j,G}}{f}-1\big)\,\mu({\rm d}f)}\right)\,\frac{{\rm d}\tau}{\tau^{1+p}}
		=c_p^{-1}\int_{\besov{p}{s}{w}}\norm{f}_{\besov{p}{s}{w}}^p\,e_S(\mu_{0,1})({\rm d}f),
	\end{align*}
	which shows the necessity of Condition \eqref{eq:BesovMeas5} because of \eqref{eq.Acosta-poly}.
	Furthermore,  since $\Besov{p}{s}{w}$ is isomorphic to $\ell^p(\mathbb{W}^d)$ and the latter is of cotype $2$ for $p\in (1,2)$, Theorem 2.2 in \cite{Araujo1978} directly shows the necessity of condition \eqref{eq:BesovMeas1}.
\end{proof}

\begin{remark}
\label{rem:TypeSuff}
 Since $\Besov{p}{s}{w}$ is of type $p$  for $p\in[1,2]$ by the isometry with $\ell^p(\mathbb{W}^d)$,  
 Theorem~2.3 in \cite{Araujo1978} guarantees that $\mu$ is L\'evy measure on $\Besov{p}{s}{w}$ if
	\begin{align*}
		\int_{\besov{p}{s}{w}}\Big(\norm{f}^p_{\besov{p}{s}{w}}\wedge 1\Big)\,\mu(\d f)<\infty.
	\end{align*}
\end{remark}

%%%%%%%%%%%%%%%%%%%%%%%%%%%%%%%%%%%%%%%%%%%%
\section{Regularisation in Besov spaces}
\label{sec:Regularisation}
Cylindrical L\'evy processes in a Banach space $U$ (as defined for example in \cite{Applebaum2010,Riedle2011,Riedle2015}) 
%as a family of continuous operators from $U^\ast$ to $L^0(\Omega;P)$ such that the $d$-dimensional projections are L\'evy processes in $\R^d$ for each $d\in\N$. \m{Is this not in the preliminaries? If not, it needs a proper introduction here. }
naturally generalise the notation of cylindrical Brownian motion, based on the theory of cylindrical measures and cylindrical random variables.  
For some $n\in\mathbb{N}$ and $u_1,\ldots,u_n \in U^\ast$ we define the projection $\pi_{u_1^\ast,\ldots,u_n^\ast}:U\to\mathbb{R}^n$ by
\begin{align*}
	\pi_{u_1^\ast,\ldots, u_n^\ast}(u)
	:= \big(\scapro{u}{u_1^\ast}_U,\ldots, \scapro{u}{u_n^\ast}_U\big).
\end{align*}
Henceforth, we shall assume $U$ is reflexive and separable, and let $\Gamma\subseteq U^\ast$.
Sets of the form
\begin{align*}
	Z(u_1^\ast,\ldots, u_n^\ast;A)
	:= \pi^{-1}_{u_1^\ast,\ldots, u_n^\ast}(A) 
	= \left\{  u\in U: \big(\scapro{u}{u_1^\ast}_U,\ldots, \scapro{u}{u_n^\ast}_U\big) \in A \right\}
\end{align*}
for $A\in\Borel(\mathbb{R}^n)$ are called \emph{cylinder sets} with respect to $(U,\Gamma)$, and the set of all cylinder sets with respect to $(U,\Gamma)$ is denoted by $\mathcal{Z}(U,\Gamma)$; we also denote $\mathcal{Z}(U,U^*)=:\mathcal{Z}(U)$. It follows that  $\mathcal{Z}(U,\Gamma)$ is an algebra.  A cylindrical random variable $Z$ in $U$ is a linear and continuous mapping $ Z\colon U^{\ast} \rightarrow L^0(\Omega,P)$. The characteristic function of $Z$ is defined by $\phi_Z (u^\ast)=E[\exp(iZu^\ast)]$ for all $u^\ast\in U^\ast$; see e.g.\ \cite{Vakhania1987}. 

\begin{definition}
A family $(L(t):\, t\ge 0)$ of cylindrical random variables $L(t)\colon U^\ast\to L^0(\Omega,P)$ is called a \emph{cylindrical L\'evy process}  if for all $u^\ast_1, ... , u^\ast_n \in U^\ast$ and $n\in \mathbb{N}$, the stochastic process $((L(t)u^\ast_1, ... , L(t)u^\ast_n):\, t \ge 0)$ is a L\'evy process in $\mathbb{R}^n$. 
\end{definition}
In order to present the L\'evy-Khintchine formula for a cylindrical L\'evy process, we must first give a definition of the cylindrical version of the L\'evy measure. 
For cylindrical L\'evy measures, we must specifically exclude sets containing the origin to avoid consistency issues with finite-dimensional projections. We define the $\pi$-system $\mathcal{Z}_\ast(U,\Gamma)\subseteq\mathcal{Z}(U,\Gamma)$ for some $\Gamma\subseteq U^\ast$ as
\begin{align*}
	\mathcal{Z}_\ast(U,\Gamma):=\{Z(u^\ast_1,\ldots,u^\ast_n;A): u_1^\ast, \dots, u_n^\ast\in \Gamma, n\in\N, A\in\Borel(\R^n), 0\notin A\}.
\end{align*} 
If $\Gamma=U^\ast$ we write $\mathcal{Z}_\ast(U):=\mathcal{Z}_\ast(U,U^\ast)$. 
A set function $\mu\colon\mathcal{Z}_\ast(U)\to[0,\infty]$ is called a \emph{cylindrical L\'evy measure} if for all $u^\ast_1, \ldots , u^\ast_n \in U^\ast$ and $n\in \mathbb{N}$ the map
	\begin{align*}
		\mu_{u^\ast_1, \ldots , u^\ast_n}\colon\Borel(\R^n)\to[0,\infty],
		\qquad
		\mu_{u^\ast_1, \ldots , u^\ast_n}(A)=\mu\circ\pi_{u^\ast_1, \ldots , u^\ast_n}^{-1}(A\setminus\{0\})
	\end{align*}
	defines a L\'evy measure on $\R^n$.
The characteristic function of a cylindrical L\'evy process $(L(t):\, t\ge 0)$ is given by
\[\phi_{L(t)} \colon U^\ast \rightarrow \mathbb{C}, \qquad \phi_{L(t)}(u^\ast)=\exp\big(t\theta_L(u^\ast)\big),\]
for all $t\ge 0$. Here,  $\theta_L \colon U^\ast \rightarrow \mathbb{C}$ is called the (cylindrical) symbol of $L$, and is of the form
\begin{align*}
	\theta_L(u^\ast) = ia(u^\ast) - \tfrac{1}{2}\scapro{u^\ast}{Qu^\ast}_{U^\ast} +\int_U\left(e^{i\scapro{u}{u^\ast}_U}-1-i\scapro{u}{u^\ast}_U \1_{B_{\mathbb{R}}}\big(\scapro{u}{u^\ast}_U\big)\right)\mu ({\rm d} u),
\end{align*}
where $a \colon U^\ast \rightarrow \mathbb{R}$ is a continuous mapping with $a(0)=0$,
the mapping $Q \colon U^\ast \rightarrow U^{\ast\ast}$ is a positive, symmetric operator  and $\mu$ is a cylindrical L\'evy measure; see \cite{Riedle2011}.

We shall analyse cylindrical L\'evy processes in $L^2(\R^d)$ to determine when they arise from a L\'evy process in a larger Besov space. This is based on the following concept:
\begin{definition}
\label{def:Induced}
	Let $H$ be a Hilbert space continuously and densely embedded in a Banach space $U$.
	A cylindrical L\'evy process $L$ in $H$ is said to be induced by a L\'evy process $Y$ in $U$ if 
	\begin{align*}
		L(t)u^\ast=\scapro{Y(t)}{u^\ast}_U \quad\text{$P$-a.s.\ for all $u^\ast\in U^\ast$ and $t\ge 0$.} 
	\end{align*} 
\end{definition}
We shall examine conditions on the cylindrical L\'evy measure of  $L$ such that $L$ is induced by a genuine L\'evy process $Y$ in $\Besov{p}{s}{w}$ for some fixed $p>1$ and $s, w\in\R$. 
With reference to Definition \ref{def:Induced}, we shall only consider Besov spaces which contain $L^2(\R^d)$. This will allow us to develop the mathematical theory without the complication arising in the case that the cylindrical L\'evy process $L$ may have a non-trivial kernel and thus may be induced by a process in a Besov space which does not contain the whole of $L^2(\R^d)$.

%%%%%%%%%%%%%%%%%%%%%%%
\subsection{Cylindrical Brownian motions}
The literature in dealing with the Besov localisation of Gaussian processes is primarily concerned with the path properties of finite-dimensional Brownian motions, or analogously with Gaussian white noise, conceived as a distribution-valued (generalised) random variable which represents the weak derivative in space and time of a finite-dimensional Brownian motion. In this work, we generalise from a distribution-valued random variable to a distribution-valued process.

Let $W$ be a standard cylindrical Brownian motion in $L^2(\R^d)$. Applying results from our previous work  \cite{Griffiths2019}, we can identify $W$ with  a distribution-valued white noise as considered in \cite{Aziznejad2018}.  
%This means that we can consider $W$ as a Gaussian white noise in $S^\ast(\R^d)$ as considered in \cite{Aziznejad2018}. 
Proposition 3.4 in \cite{Aziznejad2018} implies  that, given $p>1$ and $s,w\in\R$, the standard cylindrical Brownian motion $W$ is induced by a Brownian motion in $\Besov{p}{s}{w}$ if and only if                                    
$s<-\frac{d}{2}$ and $w<-\frac{d}{p}$.

For a general cylindrical Brownian motion $W$ in $L^2(\R^d)$  with covariance operator $Q$  and reproducing kernel Hilbert space $H_Q$ we obtain the following: suppose $p>1$ and $w,s\in\R$ are such that $L^2(\R^d)$ is continuously embeeded into $\Besov{p}{s}{w}$, then $W$ is induced by a Brownian motion in $\Besov{p}{s}{w}$ if and only if the injection $H_Q\hookrightarrow\Besov{p}{s}{w}$ is $\gamma$-Radonifying in the sense of Chapter 9 in \cite{Hytonen2017}.

%
%%%%%%%%%%%%%%%%%%%%%%%%
\subsection{Non-Gaussian cylindrical L\'evy processes}
\label{sec:nonGaussCLP}
The cylindrical L\'evy process $L=(L(t):\, t\ge 0)$ on $L^2(\R^d)$ defines by $L^\prime(t)(b^\ast):=L(t)(\iota^\ast b^\ast)$ for each $b^\ast\in\big(\Besov{p}{s}{w}\big)^\ast$  a cylindrical L\'evy process $L^\prime=(L^\prime(t):\, t\ge 0)$ on $\Besov{p}{s}{w}$ where $\iota\colon L^2(\R^d)\to\Besov{p}{s}{w}$ is the canonical embedding for some $p>1$ and $(s,w)\in E_p$.  
 
Letting $\mu$ be the  cylindrical L\'evy measure of  $L$ on $\mathcal{Z}_\ast(L^2(\R^d))$, then $\tilde{\mu}:=\mu\circ \iota^{-1}$ is the cylindrical L\'evy measure of $L^\prime$ on $\mathcal{Z}_\ast(\Besov{p}{s}{w})$;  see \cite[Th.~3.4]{Riedle2011}. 
We shall study the case when $\tilde{\mu}$ extends to a $\sigma$-finite measure on $\Borel(\Besov{p}{s}{w})$. In this case, we will mildly abuse notation and simply refer to this $\sigma$-finite and $\sigma$-additive extension as $\mu$ where the context allows no confusion. The starting point for our analysis shall be to examine this extension using the results previously developed. 

The following result demonstrates that, in the non-Gaussian case, regularisation of the cylindrical L\'evy process results from the extension of the cylindrical L\'evy measure, and furthermore allows us to concentrate in the sequel on symmetric cylindrical L\'evy processes.
\begin{theorem}
\label{thm:Symmetric}
	Let $L$ be a cylindrical L\'evy process in $L^2(\R^d)$ with no Guassian component and fix $p>1$ and $(s,w)\in E_{p}$. Let $L^S:=L-L_c$ be the symmetrisation of $L$, where $L_c$ is an independent copy of $L$. Then $L$ is induced by a L\'evy process in $\Besov{p}{s}{w}$ if and only if the cylindrical L\'evy measure of $L^S$ has an extension to a $\sigma$-finite measure on $\Borel(\Besov{p}{s}{w})$ which is a L\'evy measure on $\Besov{p}{s}{w}$.
\end{theorem}

\begin{proof}
	The `only if' implication is clear. To establish the converse, 
	denoting the cylindrical L\'evy measure of $L$ by $\mu$, we have that  $L^S$ has cylindrical L\'evy measure $\mu+\mu^-$. 
	Let $\tilde{\mu}_S$ denote the $\sigma$-additive extension of $\mu+\mu^-$ on $\Borel(\Besov{p}{s}{w})$; by assumption $\tilde{\mu}_S$ is a L\'evy measure.
	As we have $\mu\le\tilde{\mu}_S$ on each cylinder set in $\mathcal{Z}_\ast(\Besov{p}{s}{w})$, Theorem 3.4 in \cite{Riedle2015} implies that $\mu$ has a $\sigma$-additive extension on $\Borel(\Besov{p}{s}{w})$ which is a L\'evy measure.
	
	We fix $t>0$ and apply Theorem 5.6 in \cite{Riedle2015} to the function $f\colon[0,t]\to\mathcal{L}(L^2(\R^d),\Besov{p}{s}{w})$ given by $f(u)=\iota$ for all $u\in[0,t]$. We conclude that $f$ is stochastically integrable with respect to $L$ as defined in \cite{Riedle2015} since $\mu$ extends to a L\'evy measure. (This Theorem actually requires that the function $a\colon L^2(\R^d)\to\R$ as defined in Equation 3.1 of \cite{Riedle2015} is weak*-weakly sequentially continuous, however a careful analysis of the proof indicates that in a reflexive Banach space such as $\Besov{p}{s}{w}$, this requirement is not necessary.)
	It follows that the process
      $\big(\int_{[0,t]}f(u)\,\d L(u)\colon t\ge 0\big)$ forms a L\'evy process in $\Besov{p}{s}{w}$ which induces $L$.
\end{proof}
 
In the sequel, we shall make use of the following technique.
 As the sums defining membership of a given Besov space in \eqref{de.Besov} are required to be unconditionally convergent, we may take any convenient ordering of the terms.
For any enumeration of the countable set of indices $j, G$ and $m$, we denote a sum over the first $n$ terms in this enumeration by $\sum_{j,G,m}^n$.
We define for each $n\in\N$ the projection $P_n\in\mathcal{L}(\Besov{p}{s}{w})$ onto the subspace spanned by the first $n$ elements in the enumeration of $\Psi$, that is
\begin{align}
\label{eq:Projection}
	P_n f := \sum_{j,G,m}^n \Scapro{\Psi_m^{j,G}}{f}\Psi_m^{j,G},
	\qquad f\in\Besov{p}{s}{w}.
\end{align}

\begin{theorem}
\label{thm:RadonTest}
	Let $\mu$ be a cylindrical L\'evy measure on $\mathcal{Z}_\ast(\Besov{p}{s}{w})$ for some $p>1$ and $(s,w)\in E_p$. Then $\mu$ extends to a $\sigma$-finite measure on $\Borel(\Besov{p}{s}{w})$ if
	\begin{align*}
		\lim_{R\to\infty}\lim_{n\to\infty} \mu\left(\left\{f:\norm{P_n f}_{\besov{p}{s}{w}}>R\right\}\right)=0.
	\end{align*}
\end{theorem}

\begin{proof}
	We shall apply the theorem on \cite[p.327]{Gihman1974}, which gives conditions for when a cylindrical probability measure extends to a  probability measure on a Banach space with a separable dual. In order to apply this theorem in our setting, a careful study of the proof indicates that in addition to our assumptions we need to work with finite cylindrical measures satisfying the continuity condition
	\begin{align}
	\label{eq:Continuity}
		\lim_{k\to\infty} \mu\big(\{f\colon\Scapro{f}{f^\ast_{1,k}}<a_1,\ldots,\Scapro{f}{f^\ast_{m,k}}<a_m\}\big)
		=\mu\big(\{f\colon\Scapro{f}{f^\ast_{1}}<a_1,\ldots,\Scapro{f}{f^\ast_{m}}<a_m\}\big)
	\end{align}
	for Lebesgue-almost all $a_1,\ldots,a_m\in\R$ whenever $f^\ast_{i,k}\to f^\ast_i$ in $\big(\Besov{p}{s}{w}\big)^\ast$ for each $i=1,\ldots,m$. 

	Let $\{e_k\}_{k\in\N}$ be an (unconditional) Schauder basis of $\Besov{p}{s}{w}$ with coordinate functionals $\{e^\ast_k\}_{k\in\N}$ such that $\norm{e^\ast_k}_{\Besov{p'}{-s}{-w}}=1$ for each $k\in\N$.
	We consider the cylindrical measure $\mu_{1,1}$ defined by 
	\begin{align*}
		\mu_{1,1}(C)=\mu\left(C\cap\{f\colon\abs{\Scapro{f}{e^\ast_1}}>1\}\right)
   \qquad\text{for }	C\in \mathcal{Z}(\Besov{p}{s}{w},\{e^\ast_k\}_{k\in\N}).
	\end{align*}
	By the properties of cylindrical L\'evy measures, the set function $\mu_{1,1}$ is a finite cylindrical measure on $\mathcal{Z}(\Besov{p}{s}{w},\{e^\ast_k\}_{k\in\N})$. As $\mu$ satisfies $\lim_{k\to\infty}(\abs{\beta}^2\wedge 1)(\mu\circ\pi_{e^\ast_1,f^\ast_{1,k}\ldots,f^\ast_{m,k}}^{-1})(\d\beta)=(\abs{\beta}^2\wedge 1)(\mu\circ\pi_{e^\ast_1,f^\ast_1,\ldots,f^\ast_m}^{-1})(\d\beta)$ weakly for every $f^\ast_{i,k}\to f^\ast_i$ in $\big(\Besov{p}{s}{w}\big)^\ast$ for each $i=1,\ldots,m$. 
	due to Lemma 4.4 in \cite{Riedle2011}, we see that \eqref{eq:Continuity} is satisfied for $\mu_{1,1}$ as each finite-dimensional projection only takes weight on $\abs{\beta}>1$.
	Since 
		\begin{align*}
   \mu_{1,1}\left(\left\{f:\norm{P_n f}_{\besov{p}{s}{w}}>R\right\}\right)\le     \mu\left(\left\{f:\norm{P_n f}_{\besov{p}{s}{w}}>R\right\}\right)
   \qquad\text{for all }n\in\N, R>0, 
	\end{align*}
	we may apply the result  in \cite[p.327]{Gihman1974} and extend $\mu_{1,1}$ to a finite measure $\tilde{\mu}_{1,1}$ on $\Borel(\Besov{p}{s}{w})$. The measure $\tilde{\mu}_{1,1}$ is supported on $\{f\colon\abs{\Scapro{f}{e^\ast_1}}>1\}$, and satisfies $\tilde{\mu}_{1,1}(C)=\mu(C\cap\{f\colon\abs{\Scapro{f}{e^\ast_1}}>1\})$ for each cylinder set $C\in\mathcal{Z}(\Besov{p}{s}{w},\{e^\ast_k\}_{k\in\N})$.
	
	Next, for each $n\in\N$ we construct cylindrical measures $\mu_{1,n+1}$ by 
	\begin{align*}
		\mu_{1,n+1}(C)=\mu\left(C\cap\left\{f\colon\frac{1}{n+1}<\abs{\Scapro{f}{e^\ast_1}}\le\frac{1}{n}\right\}\right),
		\qquad C\in\mathcal{Z}(\Besov{p}{s}{w},\{e^\ast_k\}_{k\in\N}).
	\end{align*}
	Applying the same argument as above (using an easy rescaling), we obtain a sequence of finite measures $\{\tilde{\mu}_{1,n}\}_{n\in\N}$ with pairwise disjoint support. We define the measure $\tilde{\mu}_1$ by
	\begin{align*}
		\tilde{\mu}_1(A)=\sum_{n=1}^\infty\tilde{\mu}_{1,n}(A),
		\qquad A\in\Borel(\Besov{p}{s}{w}).
	\end{align*}
	By this construction, $\tilde{\mu}_1$ forms a $\sigma$-finite  measure on $\Borel(\Besov{p}{s}{w})$ with support in  $B_1:=\{f\colon\abs{\Scapro{f}{e^\ast_1}}\neq 0\}$. 
	
	We repeat the procedure on the subspace $\{f\colon\abs{\Scapro{f}{e^\ast_1}}=0\}$. We start by defining, for $C\in\mathcal{Z}(\Besov{p}{s}{w},\{e^\ast_k\}_{k\in\N})$,
	\begin{align*}
		\mu_{2,1}(C)=\mu\left(C\cap\{f\colon\abs{\Scapro{f}{e^\ast_1}}=0,\abs{\Scapro{f}{e^\ast_2}}>1\}\right),
	\end{align*}
	and, for $n\in\N$,
	\begin{align*}
		\mu_{2,n+1}(C)=\mu\left(C\cap\left\{f\colon\abs{\Scapro{f}{e^\ast_1}}=0,\frac{1}{n+1}<\abs{\Scapro{f}{e^\ast_2}}\le\frac{1}{n}\right\}\right).
	\end{align*}
	We in this manner obtain a $\sigma$-finite measure $\tilde{\mu}_2$ on $\Borel(\Besov{p}{s}{w})$  with support in the set $B_2:=\{f\colon\abs{\Scapro{f}{e^\ast_1}}=0,\abs{\Scapro{f}{e^\ast_2}}\neq 0\}$.
	We then continue this procedure to create the set of measures $\{\tilde{\mu}_k\}_{k\in\N}$, where for each $k\in\N$, $\tilde{\mu}_k$ has support in $B_k:=\{f\colon\abs{\Scapro{f}{e^\ast_1}}=0,\ldots,\abs{\Scapro{f}{e^\ast_{k-1}}}=0,\abs{\Scapro{f}{e^\ast_k}}\neq 0\}$. 
	We observe that the sets $\{B_k\}_{k\in\N}$ are pairwise disjoint, and
	$
		\Besov{p}{s}{w}=\{0\}\cup\bigcup_{k=1}^\infty B_k.
	$ 
	We now finally define the measure $\tilde{\mu}$ on $\Borel(\Besov{p}{s}{w})$ by setting $\tilde{\mu}(\{0\})=0$ and
	\begin{align*}
		\tilde{\mu}(A)=\sum_{k=1}^\infty\tilde{\mu}_{k}(A),
		\qquad A\in\Borel(\Besov{p}{s}{w}).
	\end{align*}
	As $\tilde{\mu}$ is a sum of $\sigma$-finite measures with pairwise disjoint support, it follows that $\tilde{\mu}$ is $\sigma$-finite.
		Let $C=C(e_1^\ast, \dots, e_n^\ast;B)$ for some $B\in\Borel(\R^n)$ with $0\notin B$ be an arbitrary set in $\mathcal{Z}_\ast(\Besov{p}{s}{w},\{e^\ast_k\}_{k\in\N})$. Since
	\begin{align*}
		\tilde{\mu}(C)=\sum_{k=1}^{n}\tilde{\mu}_{k}(C)=\sum_{k=1}^{n}\mu(C\cap B_k)=\mu(C),
	\end{align*}
   we conclude that $\tilde{\mu}$ forms an extension of the restriction of $\mu$ to $\mathcal{Z}_\ast(\Besov{p}{s}{w},\{e^\ast_k\}_{k\in\N})$. 
 %  Finally, the resulting $\sigma$-additivity of $\mu$ on $\mathcal{Z}_\ast(\Besov{p}{s}{w},\{e^\ast_k\}_{k\in\N})$ extends to $\mathcal{Z}_\ast(\Besov{p}{s}{w})$ thus showing that $\tilde{\mu}$ forms an extension of $\mu$.
 
 First assuming that $\mu$ is finite,  continuity of its characteristic functions shows that $\mu$ is uniquely determined by its values on  $\mathcal{Z}_\ast(\Besov{p}{s}{w},\{e^\ast_k\}_{k\in\N})$. It follows that $\tilde{\mu}$  forms an extension of $\mu$. The general case follows from the construction of $\tilde{\mu}$.  
\end{proof}

We now present a Corollary to Theorem \ref{thm:LevyMeasureBesov} which characterises when a cylindrical L\'evy process in $L^2(\R^d)$ is induced by a genuine L\'evy process in some given weighted Besov space.
\begin{corollary}\label{co.cylindrical-in-Besov}
Let $L$ be a cylindrical L\'evy process in $L^2(\R^d)$ with no Gaussian component and with cylindrical L\'evy measure $\mu$.  Then $L$ is induced by a L\'evy process in $\Besov{p}{s}{w}$  for some $p>1$ and $(s,w)\in E_{p}$ if and only if
$\mu$ extends to a $\sigma$-finite measure on $\Borel(\Besov{p}{s}{w})$ and 
\begin{enumerate}
\item[{\rm (1)}] for $p\ge 2$, 
\begin{align}\label{eq:CylInBes1}
&\lim_{n\to\infty} \int_{\besov{p}{s}{w}}\left(\norm{P_n f}_{\besov{p}{s}{w}}^p\wedge 1\right)\,\mu({\rm d}f)<\infty,\\
&\lim_{n\to\infty} \sum_{j,G,m}^n (\omega_m^j)^p \left(\int_{\besov{p}{s}{w}} \1_{B_{\R}}\Big(\norm{P_n f}_{\besov{p}{s}{w}}\Big)\abs{\Scapro{\Psi_m^{j,G}}{f}}^2\,\mu({\rm d}f)\right)^{p/2}<\infty;\notag
\end{align}
\item[{\rm (2)}] for $p\in (1,2)$, 
\begin{align}\label{eq:CylInBes2}
&\lim_{n\to\infty} \int_{\besov{p}{s}{w}}\left(\norm{P_n f}_{\besov{p}{s}{w}}^2\wedge 1\right)\,\mu({\rm d}f)<\infty,\\
&\lim_{n\to\infty} \sum_{j,G,m}^n (\omega_m^j)^p \int_0^\infty \left(1-e^{\int_{\besov{p}{s}{w}} \1_{B_{\R}}\big(\norm{P_n f}_{\besov{p}{s}{w}}\big)\big(\cos\tau \Scapro{\Psi_m^{j,G}}{f}-1\big)\,\mu({\rm d}f)}\right)\,\frac{{\rm d}\tau}{\tau^{1+p}}<\infty.\notag
\end{align}
\end{enumerate}
	In the expressions above, $\omega_m^j=\omega_m^j(p,s,w)$ are the weight constants defined in \eqref{eq:Weights}, 
and $\{\Psi^{j,G}_m\colon j\in\Zp, G\in G^j, m\in\Z^d\}$ is any admissible basis of $\Besov{p}{s}{w}$.
\end{corollary}
\begin{proof}
	By Theorem \ref{thm:Symmetric} we may assume that $L$ and $\mu$ are symmetric.
	Then necessity and sufficiency of the conditions follows from Theorem \ref{thm:LevyMeasureBesov}.
\end{proof}

%%%%%%%%%%%%%%%%%%%%%%%%%%%%%%
\subsection{Radonifying Embeddings}

Assume that $U$ and $V$ are some Banach spaces. 
Let $ Z\colon U^{\ast} \rightarrow L^0(\Omega,P)$ be a cylindrical random  variable and $T\colon U\to V$ a linear and continuous operator. The image of $Z$ under $T$ is the cylindrical random variable $TZ\colon V^\ast \rightarrow L^0(\Omega,P)$ on $V$ defined by $(TZ)v^\ast:=Z(T^\ast v^\ast)$ for all $v^\ast\in V^\ast$. The cylindrical random variable 
$TZ$ is induced by a random variable $X\colon \Omega\to V$ if $(TZ)v^\ast=\scapro{X}{v^\ast}_V$ $P$-a.s. for all $v^\ast\in V^\ast$. 
\begin{definition}
\label{def:Radonifying}
	Let $U,V$ be separable Banach spaces. A continuous linear operator $T\colon U\to V$ is called
	\begin{itemize}
	\item[{\rm (1)}] \textup{0-Radonifying} if for each cylindrical random variable $Z$ in $U$, the cylindrical random variable  $TZ$ is induced by a random variable in $V$;
	\item[{\rm (2)}] $p$-\textup{Radonifying} for some $p>0$ if for each cylindrical random variable $Z$ in $U$ satisfying $E\abs{Zu^\ast}^p<\infty$ for each $u^\ast\in U^\ast$, the cylindrical random variable $TZ$ is induced by a random variable in $V$.
	\end{itemize}
\end{definition}
The word $Radonifying$ originates from the fact, that the probability distribution of the inducing random variable in the definition above is a Radon measure on $\Borel(V)$. Since we only consider separable Banach spaces, this is redundant in our setting. 

\begin{definition}
	Let $U,V$ be Banach spaces and $p>0$. A continuous linear operator $T\colon U\to V$ is called $p$-\textup{summing} if there exists $C>0$ such that for every finite collection $(u_i)_{i=1}^n\subseteq U$ we have
	\begin{align*}
		\sum_{i=1}^n\norm{Tu_i}_V^p
		\le C^p \sup_{\norm{u^\ast}_{U^\ast}\le 1}\sum_{i=1}^n\abs{\scapro{u_i}{u^\ast}_U}^p. 
	\end{align*}
\end{definition}
Every $p$-Radonifying operator for $p>0$ is $p$-summing, and every 0-Radonifying operator is $p$-summing for every $p>0$. The classes of $p$-summing and $p$-Radonifying operators coincide for $p>1$;  see \cite[Th.\ VI.5.4]{Vakhania1987}. If $U$ and $V$ are Hilbert spaces, 0-Radonifying operators coincide with Hilbert-Schmidt operators \cite[Th.~VI.5.2]{Vakhania1987}, and thus the class of $p$-summing operators for all $p>0$ coincides with the class of Hilbert-Schmidt operators.

The cylindrical L\'evy process $L^\prime$, introduced in Section \ref{sec:nonGaussCLP}, is the image of $L$ under 
the embedding $\iota\colon L^2(\R^d)\to\Besov{p}{s}{w}$. If $\iota$ is $0$-radonifying, 
it follows that $L^\prime$ is automatically induced by a genuine L\'evy process on $\Besov{p}{s}{w}$, which motivates us to investigate this property in detail in the following
\begin{theorem}
\label{thm:pRadon}
The embedding $\iota\colon L^2(\R^d)\to\Besov{p}{s}{w}$ for some $p>1$ is $p$-Radonifying
if and only if 
	\begin{align}
	\label{eq:pRadonifying}
(s,w)\in  R^{(p)}_{p}:= (-\infty,-\tfrac{d}{2}) \times (-\infty,-\tfrac{d}{p}).
	\end{align}
\end{theorem}

\begin{proof}
	The continuous embedding follows as $R^{(p)}_p\subseteq E_p$. Let $\Psi$ be an admissible basis of $\Besov{p}{s}{w}$ and $\omega_m^j=\omega_m^j(p,s,w)$ be the weight constants defined in \eqref{eq:Weights}. Note, that  
\begin{align}\label{eq.w-sum-R}
	\sum_{j,G,m} (\omega_m^j)^p
	= \sum_{j,G,m}2^{jp(s-d/p+d/2)}(1+2^{-2j}\abs{m}^2)^{pw/2}<\infty
\;\Leftrightarrow \; (s,w)\in R_p^{(p)}. 
\end{align}
This follows from the fact, that the sum over $m$ converges if and only if $w<-\tfrac{d}{p}$, and in this case is asymptotically $\O(2^{jd})$ as $j\to\infty$; see e.g.\ the proof of \cite[Th.~3]{Fageot2016}. 

For any $p>1$, we obtain for $f_1,\ldots, f_n\in L^2(\R^d)$ that  
	\begin{align*}
		\sum_{i=1}^n \norm{f_i}_{\besov{p}{s}{w}}^p
	= \sum_{j,G,m} (\omega_m^j)^p \sum_{i=1}^n\abs{\scapro{f_i}{\Psi_m^{j,G}}_{L^2}}^p
		\le \sum_{j,G,m} (\omega_m^j)^p\sup_{\norm{y}_{L^{2}}\le 1}\sum_{i=1}^n\abs{\scapro{f_i}{y}_{L^2}}^p.
	\end{align*}
Thus, the embedding $L^2(\R^d)\hookrightarrow \Besov{p}{s}{w}$ is $p$-summing if $(s,w)\in R_p^{(p)}$. 
	
To establish necessity, we first consider $p\ge 2$.  Choose $f_i=\Psi_{m_i}^{j_i,G_i}\in\Psi$, $i=1,\ldots,n$ to be distinct wavelet basis vectors, i.e.\ $f_i\neq f_j$ for $i\neq j$. It follows that 
	\begin{align*}
		\sup_{\norm{y}_{L^2}\le 1}\sum_{i=1}^n\abs{\scapro{f_i}{y}_{L^2}}^p
		\le \sup_{\norm{y}_{L^2}\le 1}\sum_{j,G,m}\abs{\scapro{\Psi_m^{j,G}}{y}_{L^2}}^p
		\le \sup_{\norm{y}_{L^2}\le 1}\sum_{j,G,m}\abs{\scapro{\Psi_m^{j,G}}{y}_{L^2}}^2
		=1.
	\end{align*}
On the other side, since $\sum_{i=1}^n\norm{f_i}_{\besov{p}{s}{w}}^p=\sum_{i=1}^n(\omega_i)^p$
where $\omega_i=\omega_m^j$ when $f_i=\Psi_m^{j,G}$ for some $(j,G,m)\in\mathbb{W}^d$, 
it follows from \eqref{eq.w-sum-R} that the embedding $\iota$ is not $p$-summable if 
$(s,w)\notin R_p^{(p)}$. An application of  \cite[Th.\ VI.5.4]{Vakhania1987} completes the proof of the case $p\ge 2$. 
	
	Necessity of $(s,w)\in R_p^{(p)}$ for $p\in (1,2)$ follows from  Theorem \ref{thm:L2AlphaStable}, which gives a 
	counterexample of a cylindrical L\'evy process in $L^2(\R^d)$ with $p$-th weak moments that is not induced by a process in $\Besov{p}{s}{w}$  for any $(s,w)\notin{R^{(p)}_{p}}$. This result, though later in the text, does not rely upon the result that the embedding $L^2(\R^d)\hookrightarrow \Besov{p}{s}{w}$ is not $p$-Radonifying for $(s,w)\notin R^{(p)}_{p}$. 
\end{proof}
Due to the range of continuous embeddings between the weighted Besov spaces, it is in many cases possible to factorise the embeddings via a Hilbert space which allows for a 0-Radonification result.

\begin{theorem}
\label{thm:Radonification}
The embedding $\iota\colon L^2(\R^d)\to\Besov{p}{s}{w}$ for some $p>1$ is $0$-Radonifying
if and only if 
	\begin{align}
	\label{eq:RadonRegion}
		(s,w)\in R_{p}:=\begin{cases}
		(-\infty,-\tfrac{d}{2}) \times (-\infty,-\tfrac{d}{p}), & \text{when }p\in (1,2],\\
		(-\infty,-d+\tfrac{d}{p}) \times (-\infty,-\tfrac{d}{2}), & \text{when }p\in (2,\infty).
		\end{cases}
	\end{align}
\end{theorem}

\begin{proof}
	We show sufficiency by factorising the embedding $\iota\colon L^2(\R^d)\to\Besov{p}{s}{w}$ as follows:
	\begin{align*}
		L^2(\R^d)\overset{\iota_1} \hookrightarrow
		\Besov{2}{s_1}{w_1}\overset{\iota_2} \hookrightarrow
		\Besov{p}{s}{w},
	\end{align*}
	for some  $s_1<-\tfrac{d}{2}$ and $w_1<-\tfrac{d}{2}$. In this case, Theorem \ref{thm:pRadon} guarantees that 
	the embedding $\iota_1$ is 2-Radonifying, and thus it is $0$-radonifying, since $2$-summing operators coincide with 
	$0$-Radonifying operators in Hilbert spaces according to \cite[Th.~VI.5.2]{Vakhania1987}. It remains to show 
	that $w_1$ and $s_1$ can always be chosen, such that  $\Besov{2}{s_1}{w_1}$ is continuously embedded into $	\Besov{p}{s}{w}$, whenever $s$ and $w$ are in the stated ranges.
	
	(i) Let $p\in(1,2]$.
	By applying Proposition \ref{prop:BesovEmbeddings} we see that $\iota_2$ is a continuous embedding for $w<w_1-d\left(\tfrac{1}{p}-\tfrac{1}{2}\right)$ and $s<s_1$. Thus, by defining $\epsilon:=-w-\tfrac{d}{p}>0$, we may take $w_1=-\tfrac{\epsilon}{2}-\tfrac{d}{2}$ which satisfies both required inequalities. A similar argument may be used for $s$ and $s_1$.
	
	(ii) Let $p>2$.
	By Proposition \ref{prop:BesovEmbeddings} we see that $\iota_2$ is a continuous embedding for $w<w_1$ and $s<s_1-d\left(\tfrac{1}{2}-\tfrac{1}{p}\right)$. We proceed by similar arguments as above.	

	The necessity for $p=2$ follows from Theorem \ref{thm:pRadon} since $0$-Radonifying and $p$-Radonifying operators between Hilbert spaces coincide.
	To show the necessity for $p\neq 2$, we shall refer to Theorem \ref{thm:L2AlphaStable} and Proposition \ref{prop:AlphaStable}, which provide counterexamples of cylindrical L\'evy processes in $L^2(\R^d)$ that are not induced by a process in $\Besov{p}{s}{w}$  for any $(s,w)\notin{R_{p}}$.
	Although these negative results referred to are later in the text, they are not based upon this Theorem.
\end{proof}

\begin{remark}
	Comparing the results in Theorem \ref{thm:pRadon} and Theorem \ref{thm:Radonification}, we have $R_{p}=R^{(p)}_{p}$ for  $p\in(1,2]$ and otherwise $R_{p}\subsetneq R^{(p)}_{p}$. 
\end{remark}

We summarise the Radonification regions in Figures 4.1 and 4.2 plotting $s$ and $w$ against $\tfrac{1}{p}$, which follows naturally from the ranges specified. We refer to diagrams plotting $s$ and $w$ against $\tfrac{1}{p}$ as Triebel diagrams\footnote{Plots of $s$ and $w$ against $\tfrac{1}{p}$ were described in \cite{Aziznejad2018} as `Triebel diagrams' and used to illustrate various properties of the scales of Besov spaces.}.
\begin{figure}[h!]
\centering
  \begin{tikzpicture}[]
	%s vs 1/p for 0-Rad
   \begin{axis}[
   		scale=0.8,
   		axis lines=center,
   		xmin=0, xmax=1.1,
   		ymin=-3, ymax=3,
		xtick={0.5,1},
  		xticklabels={$\frac{1}{2}$,1},
		ytick distance=1, 
		ticklabel style={font=\small},
  		extra y ticks={-1,-2}, extra y tick labels={$-\frac{d}{2}$,$-d$},
		xlabel={$\frac{1}{p}$}, ylabel={$s$},
   		yticklabel=\empty ]

    \addplot [domain=0:1, samples=100, name path=f, thick, color=green!50]
        {min(2*x - 2,-1)};

    \addplot [domain=0:1, samples=100, name path=g, thick, color=red!50]
        {min(2*x - 1,0)};

    \path [name path=xbot]
      (axis cs:\pgfkeysvalueof{/pgfplots/xmin},\pgfkeysvalueof{/pgfplots/ymin}) --
      (axis cs:\pgfkeysvalueof{/pgfplots/xmax},\pgfkeysvalueof{/pgfplots/ymin});

    \path [name path=xtop]
      (axis cs:\pgfkeysvalueof{/pgfplots/xmin},\pgfkeysvalueof{/pgfplots/ymax}) --
      (axis cs:\pgfkeysvalueof{/pgfplots/xmax},\pgfkeysvalueof{/pgfplots/ymax});

    \addplot[orange!30, opacity=0.4] fill between[of=g and f, soft clip={domain=0:1}];
    \addplot[green!10, opacity=0.4] fill between[of=xbot and f, soft clip={domain=0:1}];
    \addplot[red!10, opacity=0.4] fill between[of=xtop and g, soft clip={domain=0:1}];

	\node[color=red, font=\footnotesize] at (axis cs: 0.5,1.5) {$E_{p}^c$};
	\node[color=orange, font=\footnotesize] at (axis cs: 0.75,-0.65) {$E_p\setminus R_{p}$};
	\node[color=green, font=\footnotesize] at (axis cs: 0.65,-2.2) {$R_{p}$};

    \end{axis}

	\end{tikzpicture}
	\qquad
	\begin{tikzpicture}[]
	%w vs 1/p for 0-Rad
   \begin{axis}[
   		scale=0.8,
   		axis lines=center,
   		xmin=0, xmax=1.1,
   		ymin=-3, ymax=3,
		xtick={0.5,1},
  		xticklabels={$\frac{1}{2}$,1},
		ytick distance=1, 
		ticklabel style={font=\small},
  		extra y ticks={-1,-2}, extra y tick labels={$-\frac{d}{2}$,$-d$},
   		xlabel={$\frac{1}{p}$}, ylabel={$w$},
   		yticklabel=\empty ]

    \addplot [domain=0:1, samples=100, name path=f, thick, color=green!50]
        {min(-2*x,-1)};

    \addplot [domain=0:1, samples=100, name path=g, thick, color=red!50]
        {min(-2*x + 1,0)};

    \path [name path=xbot]
      (axis cs:\pgfkeysvalueof{/pgfplots/xmin},\pgfkeysvalueof{/pgfplots/ymin}) --
      (axis cs:\pgfkeysvalueof{/pgfplots/xmax},\pgfkeysvalueof{/pgfplots/ymin});

    \path [name path=xtop]
      (axis cs:\pgfkeysvalueof{/pgfplots/xmin},\pgfkeysvalueof{/pgfplots/ymax}) --
      (axis cs:\pgfkeysvalueof{/pgfplots/xmax},\pgfkeysvalueof{/pgfplots/ymax});

    \addplot[orange!30, opacity=0.4] fill between[of=g and f, soft clip={domain=0:1}];
    \addplot[green!10, opacity=0.4] fill between[of=xbot and f, soft clip={domain=0:1}];
    \addplot[red!10, opacity=0.4] fill between[of=xtop and g, soft clip={domain=0:1}];

	\node[color=red, font=\footnotesize] at (axis cs: 0.5,1.5) {$E_{p}^c$};
	\node[color=orange, font=\footnotesize] at (axis cs: 0.25,-0.65) {$E_p\setminus R_{p}$};
	\node[color=green, font=\footnotesize] at (axis cs: 0.65,-2.2) {$R_{p}$};

    \end{axis}

	\end{tikzpicture}
	\caption{Triebel diagrams for 0-Radonification}
\end{figure}
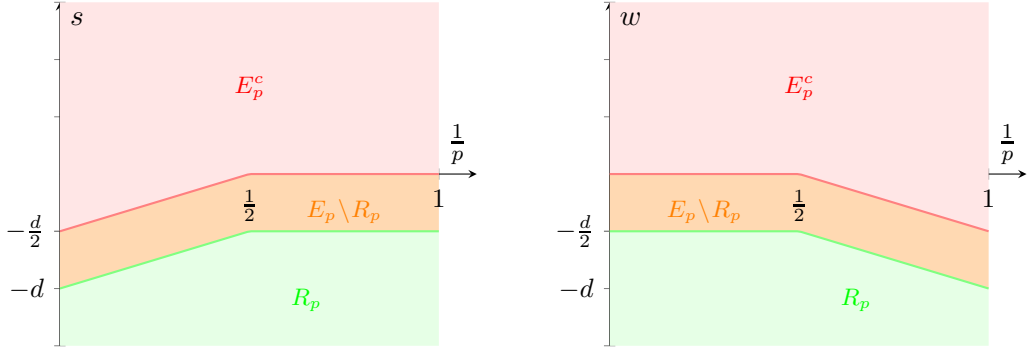
  
\begin{figure}[h!]
\centering
  \begin{tikzpicture}[]
	%s vs 1/p for p-Rad
   \begin{axis}[
   		scale=0.8,
   		axis lines=center,
   		xmin=0, xmax=1.1,
   		ymin=-3, ymax=3,
		xtick={0.5,1},
  		xticklabels={$\frac{1}{2}$,1},
		ytick distance=1, 
		ticklabel style={font=\small},
  		extra y ticks={-1,-2}, extra y tick labels={$-\frac{d}{2}$,$-d$},
		xlabel={$\frac{1}{p}$}, ylabel={$s$},
   		yticklabel=\empty ]

    \addplot [domain=0:1, samples=100, name path=f, thick, color=blue!50]
        {-1};

    \addplot [domain=0:1, samples=100, name path=g, thick, color=red!50]
        {min(2*x - 1,0)};

    \path [name path=xbot]
      (axis cs:\pgfkeysvalueof{/pgfplots/xmin},\pgfkeysvalueof{/pgfplots/ymin}) --
      (axis cs:\pgfkeysvalueof{/pgfplots/xmax},\pgfkeysvalueof{/pgfplots/ymin});

    \path [name path=xtop]
      (axis cs:\pgfkeysvalueof{/pgfplots/xmin},\pgfkeysvalueof{/pgfplots/ymax}) --
      (axis cs:\pgfkeysvalueof{/pgfplots/xmax},\pgfkeysvalueof{/pgfplots/ymax});

    \addplot[brown!30, opacity=0.4] fill between[of=g and f, soft clip={domain=0:1}];
    \addplot[blue!10, opacity=0.4] fill between[of=xbot and f, soft clip={domain=0:1}];
    \addplot[red!10, opacity=0.4] fill between[of=xtop and g, soft clip={domain=0:1}];

	\node[color=red, font=\footnotesize] at (axis cs: 0.5,1.5) {$E_{p}^c$};
	\node[color=brown, font=\footnotesize] at (axis cs: 0.75,-0.65) {$E_p\setminus R^{(p)}_{p}$};
	\node[color=blue, font=\footnotesize] at (axis cs: 0.65,-2.2) {$R^{(p)}_{p}$};

    \end{axis}

	\end{tikzpicture}
	\qquad
	\begin{tikzpicture}[]
	%w vs 1/p for p-Rad
   \begin{axis}[
   		scale=0.8,
   		axis lines=center,
   		xmin=0, xmax=1.1,
   		ymin=-3, ymax=3,
		xtick={0.5,1},
  		xticklabels={$\frac{1}{2}$,1},
		ytick distance=1, 
		ticklabel style={font=\small},
  		extra y ticks={-1,-2}, extra y tick labels={$-\frac{d}{2}$,$-d$},
   		xlabel={$\frac{1}{p}$}, ylabel={$w$},
   		yticklabel=\empty ]

    \addplot [domain=0:1, samples=100, name path=f, thick, color=blue!50]
        {-2*x};

    \addplot [domain=0:1, samples=100, name path=g, thick, color=red!50]
        {min(-2*x + 1,0)};

    \path [name path=xbot]
      (axis cs:\pgfkeysvalueof{/pgfplots/xmin},\pgfkeysvalueof{/pgfplots/ymin}) --
      (axis cs:\pgfkeysvalueof{/pgfplots/xmax},\pgfkeysvalueof{/pgfplots/ymin});

    \path [name path=xtop]
      (axis cs:\pgfkeysvalueof{/pgfplots/xmin},\pgfkeysvalueof{/pgfplots/ymax}) --
      (axis cs:\pgfkeysvalueof{/pgfplots/xmax},\pgfkeysvalueof{/pgfplots/ymax});

    \addplot[brown!30, opacity=0.4] fill between[of=g and f, soft clip={domain=0:1}];
    \addplot[blue!10, opacity=0.4] fill between[of=xbot and f, soft clip={domain=0:1}];
    \addplot[red!10, opacity=0.4] fill between[of=xtop and g, soft clip={domain=0:1}];

	\node[color=red, font=\footnotesize] at (axis cs: 0.5,1.5) {$E_{p}^c$};
	\node[color=brown, font=\footnotesize] at (axis cs: 0.75,-1) {$E_p\setminus R^{(p)}_{p}$};
	\node[color=blue, font=\footnotesize] at (axis cs: 0.65,-2.2) {$R^{(p)}_{p}$};

    \end{axis}

	\end{tikzpicture}
	\caption{Triebel diagrams for $p$-Radonification}
\end{figure}
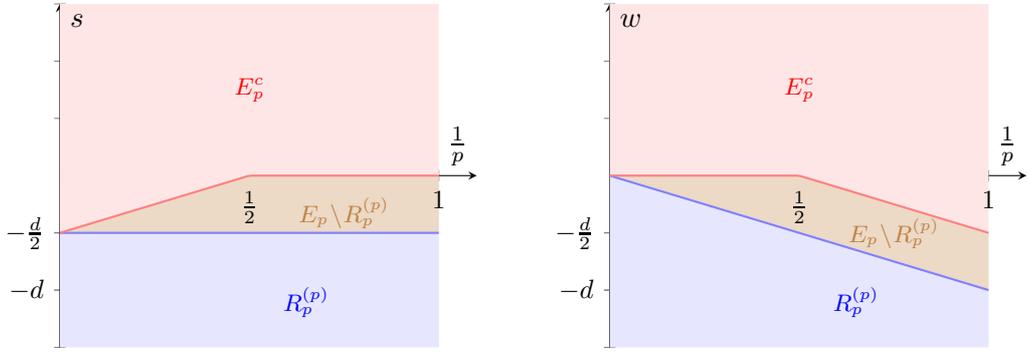

%%%%%%%%%%%%%%%%%%%%%%%%%%%%%%
\section{Examples }
\label{sec:Applications}
In this section, we apply our previous results to investigate the regularity of some typical examples of cylindrical L\'evy processes as they often appear in the literature; see e.g. \cite{Brzezniak2010,Priola2011}

\subsection{Canonical cylindrical $\alpha$-stable processes}
A cylindrical L\'evy process $L=(L(t)\colon t\ge 0)$ in $L^2(\R^d)$ is called \emph{canonical $\alpha$-stable} for some $\alpha\in(0,2)$  if its characteristic function is of the form
\begin{align*}
	\phi_{L(t)}(u)=\exp(-t\norm{u}_{L^{2}}^\alpha),
	\qquad\qquad u\in L^{2}(\R^d).
\end{align*}
The existence of a cylindrical distribution with this characteristic function is guaranteed by Bochner's theorem for cylindrical measures;
see \cite[Prop.~IV.4.2]{Vakhania1987}; two possible explicit constructions can be found in \cite{Riedle2017}. 

\begin{theorem}
\label{thm:L2AlphaStable}
	Let $L$ be a canonical $\alpha$-stable cylindrical process in $L^2(\R^d)$ 
	for some $\alpha\in(0,2)$.
	Then  $L$ is induced by a L\'evy process in $\Besov{p}{s}{w}$ for some  $p> 1$ and $(s,w)\in E_{p}$ if and only if
$s<-\frac{d}{2}$ and $w<-\frac{d}{p}$.
\end{theorem}

\begin{proof}
	Let $\Psi$ be an admissible basis of $\Besov{p}{s}{w}$. 
	We recall that for an arbitrary enumeration of the indices $(j,G,m)\in\mathbb{W}^d$, we denote a sum over the first $n$ terms in this enumeration by $\sum_{j,G,m}^n$; furthermore let $\Psi_k$ be the wavelet $\Psi_{m_k}^{j_k,G_k}$ corresponding to the $k$-th term in this enumeration.	Letting $\mu$ be the cylindrical L\'evy measure of $L$, 
	Lemma 2.4 in \cite{Riedle2017} shows that the L\'evy measure $\nu_n:=\mu\circ\pi_{\Psi_1,\ldots,\Psi_n}^{-1}$ of any $n$-dimensional projection is given by
	\begin{align*}
		\nu_n(B)=\frac{\alpha}{c_\alpha}\int_{S^n}\lambda_n({\rm d}\xi)\int_0^\infty\1_B(r\xi)r^{-1-\alpha}\,{\rm d}r
		\qquad\text{ for all }B\in\Borel(\R^n), 
	\end{align*}
	where $\lambda_n$  is uniformly distributed on the sphere $S^n=\{\xi\in\R^n\colon \abs{\xi}=1\}$ with
	\begin{align*}
		\lambda_n(S^n)=r_n:=\frac{\Gamma(\frac{1}{2})\Gamma(\frac{n+\alpha}{2})}{\Gamma(\frac{n}{2})\Gamma(\frac{1+\alpha}{2})}
	\qquad\text{and}\qquad 
		c_\alpha=
		\begin{cases}
			-\alpha\cos(\frac{\alpha\pi}{2})\Gamma(-\alpha),&\text{if }\alpha\neq 1,\\
			\frac{\alpha\pi}{2},&\text{if }\alpha=1.
		\end{cases}
	\end{align*}
	First we show sufficiency of the conditions. For $p\in(1,2]$ this follows directly from the fact that the embedding of $L^2(\R^d)\hookrightarrow \Besov{p}{s}{w}$ is 0-Radonifying according to Theorem \ref{thm:Radonification}. 
	For $p> 2$, we will establish the Borel extension of $\mu$	 by the result in Theorem \ref{thm:RadonTest} and then apply Corollary \ref{co.cylindrical-in-Besov}. We choose an enumeration of the indices $j, G$ and $m$ and observe that  the projection $P_n$ onto the subspace spanned by the first $n$ elements  in the enumeration of $\Psi$, defined in  \eqref{eq:Projection}, satisfies
\begin{align*}
	\norm{P_n f}_{\besov{p}{s}{w}}^p
	=\sum_{j,G,m}^n (\omega_m^j)^p\abs{\Scapro{\Psi_m^{j,G}}{f}}^p,
	\qquad f\in\Besov{p}{s}{w}.
\end{align*}
We obtain for each $n\in\N$ that 
	\begin{align}
	\label{eq:Sigman}
		\Sigma_n &:=\int_{\besov{p}{s}{w}}\left(\norm{P_n f}_{\besov{p}{s}{w}}^p\wedge 1\right)\,\mu({\rm d}f)\notag\\
		&=\int_{\R^n}\left(\sum_{j,G,m}^n (\omega_m^j)^p\abs{\beta_m^{j,G}}^p\wedge 1\right)\,(\mu\circ\pi^{-1}_{\Psi_1,\ldots,\Psi_n})({\rm d}\beta)\notag\\
		&=\frac{\alpha}{c_\alpha} \int_{S^n}\int_0^\infty\left(\sum_{j,G,m}^n (\omega_m^j)^p\abs{r\xi_m^{j,G}}^p\wedge 1\right)r^{-1-\alpha}\,{\rm d}r \, \lambda_n({\rm d}\xi)\notag\\
		&=\frac{p}{c_\alpha(p-\alpha)} \int_{S^n}\left(\sum_{j,G,m}^n (\omega_m^j)^p\abs{\xi_m^{j,G}}^p\right)^{\alpha/p}\,\lambda_n({\rm d}\xi).
	\end{align}
	Letting $\lambda_n^1:=\tfrac{1}{r_n}\lambda_n$,  Jensen's inequality implies
	\begin{align*}
		\Sigma_n &=\frac{pr_n}{c_\alpha(p-\alpha)} \int_{S^n}\left(\sum_{j,G,m}^n (\omega_m^j)^p\abs{\xi_m^{j,G}}^p\right)^{\alpha/p}\,\lambda_n^1({\rm d}\xi)\\
		&\le \frac{pr_n}{c_\alpha(p-\alpha)} \left(\int_{S^n}\sum_{j,G,m}^n (\omega_m^j)^p\abs{\xi_m^{j,G}}^p\,\lambda_n^1({\rm d}\xi)\right)^{\alpha/p}.
	\end{align*}
	By Lemma A.2 in \cite{Riedle2017} we have $\int_{S^n}\abs{\xi_m^{j,G}}^p\,\lambda_n^1({\rm d}\xi)=\tfrac{\Gamma(\frac{n}{2})\Gamma(\frac{1+p}{2})}{\Gamma(\frac{1}{2})\Gamma(\frac{n+p}{2})}$ and thus
	\begin{align*}
		\Sigma_n &\le\frac{pr_n}{c_\alpha(p-\alpha)} \left(\sum_{j,G,m}^n (\omega_m^j)^p \frac{\Gamma(\frac{n}{2})\Gamma(\frac{1+p}{2})}{\Gamma(\frac{1}{2})\Gamma(\frac{n+p}{2})} \right)^{\alpha/p}\\
		&= \frac{p}{c_\alpha(p-\alpha)} \frac{\Gamma(\frac{1}{2})\Gamma(\frac{n+\alpha}{2})}{\Gamma(\frac{n}{2})\Gamma(\frac{1+\alpha}{2})} \left(\frac{\Gamma(\frac{n}{2})\Gamma(\frac{1+p}{2})}{\Gamma(\frac{1}{2})\Gamma(\frac{n+p}{2})} \right)^{\alpha/p}\left(\sum_{j,G,m}^n (\omega_m^j)^p \right)^{\alpha/p}.
	\end{align*}
	Since $\frac{\Gamma(x+\alpha)}{\Gamma(x)}=x^\alpha\big(1+\O(x^{-1}))$ as $x\to\infty$ \cite[Prop.~2.1.3]{Beals2010},  
	we conclude that  $\Sigma_n$ converges to a finite limit $\Sigma_\infty$ as $n\to\infty$ since $\sum_{j,G,m} (\omega_m^j)^p<\infty$
	given $s<-\tfrac{d}{2}$ and $w<-\tfrac{d}{p}$ (see the proof of \cite[Th.~3]{Fageot2016}).  
	Next we derive 
	\begin{align*}
		\Upsilon_n := &\sum_{j,G,m}^n (\omega_m^j)^p \left(\int_{\besov{p}{s}{w}} \1_{B_{\R}}\big(\norm{P_n f}_{\besov{p}{s}{w}}\big)\abs{\Scapro{\Psi_m^{j,G}}{f}}^2\,\mu({\rm d}f)\right)^{p/2}\\
		= &\sum_{j,G,m}^n (\omega_m^j)^p \left(\int_{\R^n} \1_{B_{\R}}\left(\sum_{i,H,l}^n (\omega_l^i)^p\abs{\beta_l^{i,H}}^p\right)\abs{\beta_m^{j,G}}^2\,(\mu\circ\pi^{-1}_{\Psi_1,\ldots ,\Psi_n})({\rm d}\beta)\right)^{p/2}\\
		= &\sum_{j,G,m}^n (\omega_m^j)^p \left(\frac{\alpha}{c_\alpha}\int_{S^n}\int_0^\infty \1_{B_{\R}}\left(\sum_{i,H,l}^n (\omega_l^i)^p\abs{r\xi_l^{i,H}}^p\right)\abs{r\xi_m^{j,G}}^2 r^{-1-\alpha}\,{\rm d}r\,\lambda_n({\rm d}\xi)\right)^{p/2}\\
		= &\left(\frac{\alpha}{c_\alpha(2-\alpha)}\right)^{p/2} \sum_{j,G,m}^n (\omega_m^j)^p \left(\int_{S^n}\abs{\xi_m^{j,G}}^2\left(\sum_{i,H,l}^n (\omega_l^i)^p\abs{\xi_l^{i,H}}^p\right)^{(\alpha-2)/p}\,\lambda_n({\rm d}\xi)\right)^{p/2}.
	\end{align*}
	Applying first Jensen's inequality and then H\"older's inequality, we obtain
	\begin{align*}
		\Upsilon_n &\le \left(\frac{\alpha}{c_\alpha(2-\alpha)}\right)^{p/2} r_n^{p/2} \sum_{j,G,m}^n (\omega_m^j)^p \int_{S^n}\abs{\xi_m^{j,G}}^p\left(\sum_{i,H,l}^n (\omega_l^i)^p\abs{\xi_l^{i,H}}^p\right)^{(\alpha-2)/2}\,\lambda_n^1({\rm d}\xi)\\
		 & =\left(\frac{\alpha}{c_\alpha(2-\alpha)}\right)^{p/2} r_n^{p/2} \int_{S^n}\left(\sum_{j,G,m}^n (\omega_m^j)^p \abs{\xi_m^{j,G}}^p\right)^{\alpha/2}\,\lambda_n^1({\rm d}\xi)\\
		& \le\left(\frac{\alpha}{c_\alpha(2-\alpha)}\right)^{p/2} r_n^{p/2} \left(\int_{S^n} \sum_{j,G,m}^n (\omega_m^j)^p \abs{\xi_m^{j,G}}^p\,\lambda_n^1({\rm d}\xi)\right)^{\alpha/2}\\
		& =\left(\frac{\alpha}{c_\alpha(2-\alpha)}\right)^{p/2} r_n^{p/2} \left(\sum_{j,G,m}^n (\omega_m^j)^p\right)^{\alpha/2}\left(\frac{\Gamma(\frac{n}{2})\Gamma(\frac{1+p}{2})}{\Gamma(\frac{1}{2})\Gamma(\frac{n+p}{2})} \right)^{\alpha/2}.
	\end{align*}
	Since by properties of the Gamma function we have 
	\begin{align*}
		r_n^{p/2} \left(\frac{\Gamma(\frac{n}{2})\Gamma(\frac{1+p}{2})}{\Gamma(\frac{1}{2})\Gamma(\frac{n+p}{2})} \right)^{\alpha/2}
		= \left(\frac{\Gamma(\frac{1}{2})\Gamma(\frac{n+\alpha}{2})}{\Gamma(\frac{n}{2})\Gamma(\frac{1+\alpha}{2})} \right)^{p/2} \left(\frac{\Gamma(\frac{n}{2})\Gamma(\frac{1+p}{2})}{\Gamma(\frac{1}{2})\Gamma(\frac{n+p}{2})} \right)^{\alpha/2}
		=\O(1)\quad\text{as }n\to\infty,
	\end{align*}
	it follows that $\Upsilon_n$ has a finite limit as $n\to\infty$ as $\sum_{j,G,m}^n (\omega_m^j)^p$ does. 

    To verify the condition in Theorem \ref{thm:RadonTest}, we fix $R>0$ and conclude 
	\begin{align*}
		\mu\left(\left\{f:\norm{P_n f}_{\besov{p}{s}{w}}^p>R^p\right\}\right)
		&=\int_{\R^n} \1_{\Big\{\sum_{j,G,m}^n (\omega_m^j)^p\abs{\beta_m^{j,G}}^p>R^p\Big\}}(\beta)\,(\mu\circ\pi^{-1}_{\Psi_1,\ldots,\Psi_n})({\rm d}\beta)\\
		&=\frac{\alpha}{c_\alpha} \int_{S^n}\int_{\sum_{j,G,m}^n(\omega_m^j)^p\abs{r\xi_m^{j,G}}^p>R^p}^\infty r^{-1-\alpha}\,{\rm d}r \, \lambda_n({\rm d}\xi)\\
		&=R^{-\alpha}c_\alpha^{-1}\int_{S^n}\left(\sum_{j,G,m}^n (\omega_m^j)^p\abs{\xi_m^{j,G}}^p\right)^{\alpha/p}\,\lambda_n({\rm d}\xi)\\
		&=R^{-\alpha}\frac{p-\alpha}{p}\Sigma_n.
	\end{align*}
	As $\Sigma_n\to\Sigma_\infty<\infty$ as $n\to\infty$, we obtain $\lim_{R\to\infty}\lim_{n\to\infty}\mu\big(\{f:\norm{P_n f}_{\besov{p}{s}{w}}>R\}\big)=0$. Theorem \ref{thm:RadonTest} implies that $\mu$ can be extended to a $\sigma$-finite measure on $\Besov{p}{s}{w}$. As we have shown $\lim_{n\to\infty}\Sigma_n <\infty$ and $\lim_{n\to\infty}\Upsilon_n <\infty$, applying Corollary \ref{co.cylindrical-in-Besov}  completes the proof for sufficiency. 
	
	We now show the necessity of the conditions in the hypothesis. First we consider the case $p\ge 2$.
	We define $A_n:=\sum_{j,G,m}^n (\omega_m^j)^p$ and, applying Jensen's inequality to the concave 
	function $x\mapsto x^{\alpha/p}$, we obtain from \eqref{eq:Sigman}, using again \cite[Lem.~A.2]{Riedle2017},
	\begin{align*}
		\Sigma_n 
		&=\frac{pA_n^{\alpha/p}}{c_\alpha(p-\alpha)} \int_{S^n}\left(\sum_{j,G,m}^n A_n^{-1} (\omega_m^j)^p\abs{\xi_m^{j,G}}^p\right)^{\alpha/p}\,\lambda_n({\rm d}\xi)\\
		&\ge \frac{pA_n^{\alpha/p}}{c_\alpha(p-\alpha)} \int_{S^n}\sum_{j,G,m}^n A_n^{-1} (\omega_m^j)^p\abs{\xi_m^{j,G}}^\alpha\,\lambda_n({\rm d}\xi)
%		&= \frac{pA_n^{\alpha/p}}{c_\alpha(p-\alpha)}r_n\frac{\Gamma(\frac{n}{2})\Gamma(\frac{1+\alpha}{2})}{\Gamma(\frac{1}{2})\Gamma(\frac{n+\alpha}{2})}
		= \frac{pA_n^{\alpha/p}}{c_\alpha(p-\alpha)}.
	\end{align*}
	For $s\ge-\tfrac{d}{2}$ or $w\ge-\tfrac{d}{p}$ we have $A_n\to\infty$ as $n\to\infty$ (see the proof of \cite[Th.~3]{Fageot2016}). It follows $\Sigma_n\to\infty$ as $n\to\infty$, showing that Condition \eqref{eq:CylInBes1} is not verified. 
		
	For $p\in(1,2)$ we calculate as above for $\Sigma_n$ that 
	\begin{align*}
		\Lambda_n &:=\int_{\besov{p}{s}{w}}\left(\norm{P_n f}_{\besov{p}{s}{w}}^2\wedge 1\right)\,\mu({\rm d}f)
%		&=\int_{\R^n}\left(\left(\sum_{j,G,m}^n (\omega_m^j)^p\abs{\beta_m^{j,G}}^p\right)^{2/p}\wedge 1\right)\,(\mu\circ\pi^{-1}_{\Psi_1,\ldots\Psi_n})({\rm d}\beta)\\
%		&=\frac{\alpha}{c_\alpha} \int_{S^n}\int_0^\infty\left(\left(\sum_{j,G,m}^n (\omega_m^j)^p\abs{r\xi_m^{j,G}}^p\right)^{2/p}\wedge 1\right)r^{-1-\alpha}\,{\rm d}r\lambda_n({\rm d}\xi)\\
		=\frac{2}{c_\alpha(2-\alpha)} \int_{S^n}\left(\sum_{j,G,m}^n (\omega_m^j)^p\abs{\xi_m^{j,G}}^p\right)^{\alpha/p}\,\lambda_n({\rm d}\xi).
	\end{align*}
	    For $p\ge\alpha$, we can conclude as above for $\Sigma_n$ that $\Lambda_n\to\infty$ as $n\to\infty$ for $s\ge-\tfrac{d}{2}$ or $w\ge-\tfrac{d}{p}$. It follows that 
	    Condition \eqref{eq:CylInBes2}  is not met for these values of $s$ and $w$, which shows the necessity.
        For $p<\alpha$, applying Jensen's inequality   to the convex  function $x\mapsto x^{\alpha/p}$
          gives us
	\begin{align*}
		\Lambda_n 
		&=\frac{2r_n}{c_\alpha(2-\alpha)}\int_{S^n}\bigg(\sum_{j,G,m}^n (\omega_m^j)^p\abs{\xi_m^{j,G}}^p\bigg)^{\alpha/p}\,\lambda_n^1({\rm d}\xi)\\
		&\ge \frac{2r_n}{c_\alpha(2-\alpha)}\bigg(\int_{S^n}\sum_{j,G,m}^n (\omega_m^j)^p\abs{\xi_m^{j,G}}^p\,\lambda_n^1({\rm d}\xi)\bigg)^{\alpha/p}
		= \frac{2r_n}{c_\alpha(2-\alpha)}\bigg(\sum_{j,G,m}^n (\omega_m^j)^p \frac{\Gamma(\frac{n}{2})\Gamma(\frac{1+p}{2})}{\Gamma(\frac{1}{2})\Gamma(\frac{n+p}{2})} \bigg)^{\alpha/p}. 
	\end{align*}
	Since $\frac{\Gamma(x+\alpha)}{\Gamma(x)}=x^\alpha\big(1+\O(x^{-1}))$ as $x\to\infty$, see \cite[Prop.~2.1.3]{Beals2010},
 and  $\sum_{j,G,m} (\omega_m^j)^p=\infty$,  it follows $\Lambda_n\to\infty$ as $n\to\infty$, showing that	    Condition \eqref{eq:CylInBes2} in Corollary \ref{co.cylindrical-in-Besov} is not met.
\end{proof}

 Recall that $R_{p}$ denotes the $(s,w)$-plane where the embeddings $L^2(\R^d)\hookrightarrow\Besov{p}{s}{w}$ are 0-Radonifying. 
	For $p\le 2$, Theorem \ref{thm:L2AlphaStable} states that $L$ is induced by a L\'evy process in $\Besov{p}{s}{w}$  if and only if $(s,w)\in R_{p}$.  However, for $p>2$ the Theorem gives a stronger result, in that the region of the $(s,w)$-plane where $L$ is induced by a L\'evy process in $\Besov{p}{s}{w}$  is a proper superset of $R_{p}$.

%%%%%%%%%%%%%%%%%%%%%%%%%%%
\subsection{Hedgehog cylindrical L\'evy process}

In this section let $L$ be a cylindrical L\'evy process in $L^2(\R^d)$ of the form 
\begin{align}\label{eq.hedgehog}
L(t)f=\sum_{k=1}^\infty \Scapro{f}{e_k} a_k \ell_k
\qquad\text{for all }f\in L^2(\R^d), \, t\ge 0, 
\end{align}
where $(e_k)_{k\in\N}$ is an orthonormal basis of $L^2(\R^d)$ and $(\ell_k)_{k\in\N}$ are identically distributed and independent symmetric real-valued L\'evy processes with characteristics $(0,0,\rho)$ for a L\'evy measure $\rho\neq 0$ in $\R$. By Theorem \ref{thm:Symmetric} it is sufficient for our analysis to focus on the symmetric case.
The sequence $(a_k)_{k\in\N}$  is real-valued and satisfies
\begin{align}
\label{eq:HedgehogWeightSum}
	\sum_{k=1}^\infty \int_{\R}\big(\abs{a_k c_k \beta}^2\wedge 1\big)\,\rho(\d \beta)<\infty
\end{align}
for each $(c_k)_{k\in\N}\in\ell^2(\R)$.
This condition guarantees that the sum in \eqref{eq.hedgehog} converges $P$-a.s.\ in $\R$; see  \cite[Lem.~4.2]{Riedle2015}. 
%To avoid redundancies in the representation, we assume $a_k\neq 0$ for all $k\in\N$. \m{I realised this is a strange condition. Do we need this anywhere?}

The support of the cylindrical L\'evy measure $\mu$ of $L$ is in $\bigcup_{k\in\N}\{\beta e_k:\, \beta\in\R\}$, as $(\ell_k)_{k\in\N}$ are independent, that is to say the measure only has weight on the axes. For this reason, 
 we refer to this process as a \emph{hedgehog cylindrical process}. 

We first present further corollaries to Theorem \ref{thm:LevyMeasureBesov} and Remark \ref{rem:TypeSuff} tailored to this setting.
\begin{corollary} 
\label{cor:HedgehogLevyMeasure}
Let $L$ be a cylindrical L\'evy process of the form as in \eqref{eq.hedgehog}, where
 $\{e_k\}_{k\in\N}\subseteq\Besov{p'}{-s}{-w}$ %\m{where do we need this?}
for some $p>1$ and $(s,w)\in E_{p}$.
Then $L$ is induced by a L\'evy process in $\Besov{p}{s}{w}$ if and only if 
\begin{enumerate}
\item[{\rm (1)}] for $p\ge 2$, 
\begin{align}
&\sum_{k=1}^\infty \int_{\R} \left( \norm{a_ke_k}_{\besov{p}{s}{w}}^p \abs{\beta}^p \wedge 1\right)\, \rho(\d \beta)<\infty, \label{eq.hedgehog-cond-1}\\
&\sum_{j,G,m} (\omega_m^j)^p\Bigg(\sum_{k=1}^\infty\abs{\Scapro{\Psi_m^{j,G}}{a_k e_k}}^2\int_{\abs{\beta}\le\norm{a_k e_k}_{\besov{p}{s}{w}}^{-1}}\beta^2\,\rho({\rm d}\beta)\Bigg)^{p/2}<\infty;\label{eq.hedgehog-cond-2}
\end{align}
\item[{\rm (2)}] for $p\in (1,2)$, 
\begin{align}
&\sum_{k=1}^\infty \int_{\R} \left( \norm{a_ke_k}_{\besov{p}{s}{w}}^2 \abs{\beta}^2 \wedge 1\right)\, \rho(\d \beta)<\infty, \label{eq.hedgehog-cond-3}\\
& \sum_{j,G,m} (\omega_m^j)^p \int_0^\infty \left(1-e^{\sum_{k=1}^\infty\int_{\beta\le\norm{a_k e_k}_{\besov{p}{s}{w}}^{-1}}(\cos\tau\Scapro{\Psi_m^{j,G}}{a_k e_k}\beta-1)\,\rho({\rm d}\beta)}\right)\,\frac{{\rm d}\tau}{\tau^{1+p}}<\infty.\label{eq.hedgehog-cond-4}
\end{align}
\end{enumerate}
\end{corollary}
\begin{proof}
	Let $\Psi$ be an admissible basis for $\Besov{p}{s}{w}$. Lemma 4.2 in \cite{Riedle2015} and Lemma 3.10 in \cite{Kosmala2020} show that  the cylindrical L\'evy measure $\mu$ of $L$ extends to a Borel measure on $L^2(\R^d)$ %\m{as discussed this is not the same as you obtain from Theorem \ref{thm:RadonTest}}
	with projection on the $n$-th partial sum given by
	\begin{align}
	\label{eq:HedgehogPartialMeasure}
		(\mu\circ\pi_{e_1, \ldots, e_n}^{-1})({\rm d}\beta_1\cdots{\rm d}\beta_n)
		=\sum_{k=1}^n (\rho\circ m_{a_k})(\d \beta_k),
	\end{align}
	where $m_{a_k}\colon \R\to\R$ is defined by $m_{a_k}(\beta)=a_k\beta$. 
	As $\{e_k\}_{k\in\N}\subseteq\Besov{p'}{-s}{-w}$, the same relationship holds when we consider the pushforward of $\mu$ to $\Besov{p}{s}{w}$.
	%We henceforth identify $\mu$ with its pushforward measure on $\Besov{p}{s}{w}$ under the canonical injection.
   Because of Theorem \ref{thm:Symmetric} it is sufficient to show that the claimed conditions are equivalent to 
   the conditions in Theorem \ref{thm:LevyMeasureBesov}. 

	For $p>1$ we define $q:=p\vee 2$ and calculate for each $n\in\N$, applying \eqref{eq:HedgehogPartialMeasure},
	\begin{align*}
		&\int_{\besov{p}{s}{w}}\left(\left(\sum_{j,G,m}\abs{\omega_m^j\sum_{k=1}^n\Scapro{\Psi_m^{j,G}}{e_k}\Scapro{e_k}{f}}^p\right)^{q/p}\wedge 1\right)\,\mu({\rm d}f)\\
		&=\int_{\R^n}\left(\left(\sum_{j,G,m}\abs{\omega_m^j\sum_{k=1}^n\Scapro{\Psi_m^{j,G}}{e_k}\beta_k}^p\right)^{q/p}\wedge 1\right)\,(\mu\circ\pi_{e_1,\ldots ,e_n}^{-1})({\rm d}\beta_1\cdots{\rm d}\beta_n)\\
		&= \sum_{k=1}^n \int_{\R}\left(\left(\sum_{j,G,m}\abs{\omega_m^j\Scapro{\Psi_m^{j,G}}{e_k}a_k\beta}^p\right)^{q/p}\wedge 1\right)\,\rho({\rm d}\beta)\\
		&= \sum_{k=1}^n \int_{\R}\Big(\norm{a_k e_k}_{\besov{p}{s}{w}}^q\abs{\beta}^q\wedge 1\Big)\,\rho({\rm d}\beta).
	\end{align*}
By taking the limit as $n\to\infty$  we obtain 
 \begin{align*}
 	\int_{\besov{p}{s}{w}} \left(\norm{f}_{\besov{p}{s}{w}}^q \wedge 1\right)\, \mu(\d f)
 	=\sum_{k=1}^\infty \int_{\R} \left( \norm{a_ke_k}_{\besov{p}{s}{w}}^q \abs{\beta}^q \wedge 1\right)\, \rho(\d \beta),
 \end{align*}
which shows the equivalence between the Conditions \eqref{eq.hedgehog-cond-1} and \eqref{eq.hedgehog-cond-3}  and Conditions \eqref{eq:BesovMeas3} and  \eqref{eq:BesovMeas1}. 

For $p\ge 2$ we calculate, for $(i,H,l)\in\mathbb{W}^d$ and $n\in\N$, that 
	\begin{align*}
		&\int_{\R^n}\1_{B_{\R}}\left(\sum_{j,G,m}(\omega_m^j)^p \abs{\sum_{k=1}^n\Scapro{\Psi_m^{j,G}}{e_k} \beta_k}^p\right)
		\abs{\sum_{k=1}^n\Scapro{\Psi_l^{i,H}}{e_k}\beta_k}^2\,(\mu\circ\pi_{e_1,\ldots ,e_n}^{-1})({\rm d}\beta_1\cdots{\rm d}\beta_n)\\
		&=\sum_{k=1}^n\int_{\R}\1_{B_{\R}}\left(\sum_{j,G,m}(\omega_m^j)^p \abs{\Scapro{\Psi_m^{j,G}}{e_k} a_k\beta}^p\right)
		\abs{\Scapro{\Psi_l^{i,H}}{e_k}a_k\beta}^2\,\rho({\rm d}\beta)\\
		&=\sum_{k=1}^n \abs{\Scapro{\Psi_l^{i,H}}{a_k e_k}}^2\int_{\abs{\beta}\le\norm{a_k e_k}_{\besov{p}{s}{w}}^{-1}}  \beta^2\,\rho({\rm d}\beta).
	\end{align*}
Since the theorem on monotone convergence shows 
\begin{align*}
&	\sum_{j,G,m} (\omega_m^j)^p \left(\int_{\norm{f}_{\besov{p}{s}{w}}\le 1} \Scapro{\Psi_m^{j,G}}{f}^2\,\mu({\rm d}f)\right)^{p/2}\\
&\qquad\qquad	=\sum_{j,G,m} (\omega_m^j)^p\Bigg(\sum_{k=1}^\infty\abs{\Scapro{\Psi_m^{j,G}}{a_k e_k}}^2\int_{\abs{\beta}\le\norm{a_k e_k}_{\besov{p}{s}{w}}^{-1}}\beta^2\,\rho({\rm d}\beta)\Bigg)^{p/2},
\end{align*}
we obtain the equivalence between Conditions \eqref{eq.hedgehog-cond-2}  and \eqref{eq:BesovMeas2}. 
A similar calculations shows the  equivalence between Conditions \eqref{eq.hedgehog-cond-4}  and \eqref{eq:BesovMeas5}. 
\end{proof}

\begin{corollary} 
	\label{cor:HedgehogTypeSuff}
Let $p\in [1,2]$ and $(s,w)\in E_{p}$. 
A cylindrical L\'evy process $L$ of the form \eqref{eq.hedgehog} with $\{e_k\}_{k\in\N}\subseteq\Besov{p'}{-s}{-w}$ is induced by a L\'evy process in $\Besov{p}{s}{w}$  if 
	\begin{align*}
		\sum_{k=1}^\infty \int_{\R} \left( \norm{a_ke_k}_{\besov{p}{s}{w}}^p \abs{\beta}^p \wedge 1\right)\, \rho(\d \beta)<\infty.
	\end{align*}
\end{corollary}

\begin{proof}
	By Remark \ref{rem:TypeSuff} it suffices to show
$
		\int_{\Besov{p}{s}{w}}\Big(\norm{f}_{\besov{p}{s}{w}}^p\wedge 1\Big)\,\mu({\rm d}f)<\infty,
$
	which follows using the same calculation as in the proof of Corollary \ref{cor:HedgehogLevyMeasure}. 
\end{proof}

To characterise the Besov membership of a hedgehog process $L$ we introduce some indices in terms of the L\'evy measure of the  real-valued L\'evy processes $\ell_k$ in the representation \eqref{eq.hedgehog}. For this purpose, let $\rho$ be a L\'evy measure in $\R$ and define for $q\in\Rp$:
\begin{align}
\label{eq:DefSupTau}
	\overline{\tau}^{(q)}&:=\inf_{\tau\ge 0}\left\lbrace \limsup_{\xi\downarrow 0}\xi^{-\tau}\int_{B_{\R}^c}\big(\xi^q\abs{\beta}^{q}\wedge 1\big)\,\rho({\rm d}\beta)=\infty \right\rbrace ,\\
\label{eq:DefInfTau}
	\underline{\tau}^{(q)}&:=\inf_{\tau\ge 0}\left\lbrace \liminf_{\xi\downarrow 0}\xi^{-\tau}\int_{B_{\R}^c}\big(\xi^q\abs{\beta}^{q}\wedge 1\big)\,\rho({\rm d}\beta)=\infty \right\rbrace .
\end{align}
In all definitions above we apply the convention $\sup \emptyset =-\infty$ and $\inf \emptyset =\infty$. 
It is easy to see that $\overline{\tau}^{(q)}\le\underline{\tau}^{(q)}\le q$ when $\rho\neq 0$. The examples following Theorem \ref{thm:Hedgehog} show calculations of these indices in a number of standard situations.

The following proposition establishes a simple interpretation of $\overline{\tau}^{(q)}$. We recall that a L\'evy process with L\'evy measure $\rho$ has finite $p$-th moments if and only if $\int_{B_{\R}^c}\abs{\beta}^{p}\,\rho({\rm d}\beta)<\infty$.
\begin{proposition}
\label{prop:HedgeIndexCons}
For a L\'evy measure $\rho\neq 0$ on $\Borel(\R)$ define 
	\begin{align*}
		p_{\max}:=\sup\left\{p>0\colon \int_{B_{\R}^c}\abs{\beta}^{p}\,\rho({\rm d}\beta)<\infty\right\}. 
	\end{align*}
If $p_{\max}>0$ then $\overline{\tau}^{(q)}=p_{\max}\wedge q$ and thus $(p_{\max}\wedge q)\le\underline{\tau}^{(q)}\le q$ for all $q\ge 0$.
%\footnote{This index $p_{\max}$ is the same as defined in \cite{Aziznejad2018}.}.
	
\end{proposition}
\begin{proof} 
		
To demonstrate this, we shall consider the following indices:
	\begin{align*}
		\overline{\tau}_1^{(q)}&:=\sup\left\lbrace \tau\ge 0\colon \limsup_{\xi\downarrow 0}\xi^{-\tau}\int_{1<\abs{\beta}\le\xi^{-1}}\xi^q\abs{\beta}^q\,\rho({\rm d}\beta)<\infty \right\rbrace \qquad \text{for }	q\in\Rp,\\
		\overline{\tau}_2&:=\sup\left\lbrace \tau> 0\colon \limsup_{\xi\downarrow 0}\xi^{-\tau}\int_{\abs{x}>\xi^{-1}}\rho({\rm d}\beta)<\infty \right\rbrace.
	\end{align*}
	We define a finite measure $\overline{\rho}:=\restr{\rho}{B_{\R}^c}$; clearly we may replace $\rho$ with $\overline{\rho}$ in the definitions of $\overline{\tau}_1^{(q)}$ and $\overline{\tau}_2$. By Markov's inequality we have, for $\xi<1$ and $p<p_{\max}$,
	\begin{align*}
		\overline{\rho}\big(\{\abs{\beta}>\xi^{-1}\}\big)
		\le \xi^{p}\int_{\R}\abs{\beta}^p\,\overline{\rho}(\d \beta)
		=\xi^{p}\int_{B_{\R}^c}\abs{\beta}^p\,\rho(\d \beta),
	\end{align*}
	thus showing $\overline{\tau}_2\ge p_{\max}$. On the other side,  for $\tau<\overline{\tau}_2$ 
	there exists a constant $C>0$ such that
	\begin{align*}
		\overline{\rho}\big(\{\abs{\beta}>t\}\big)\le C t^{-\tau}
		 \qquad\text{for all }t\ge 1.
	\end{align*}
	The tail formula for the integral shows for $0<p<\tau$ that 
	\begin{align*}
		\int_{B_{\R}^c}\abs{\beta}^p\,\rho(\d \beta)
		=\int_{\R}\abs{\beta}^p\,\overline{\rho}(\d \beta)
		=\int_0^\infty \overline{\rho}\big(\{\abs{\beta}^p>t\}\big)\,\d t
		\le\overline{\rho}(\R)+ C\int_1^\infty t^{-\frac{\tau}{p}}\,\d t<\infty,
	\end{align*}
	which enables us to conclude  $p_{\max}\ge\overline{\tau}_2$,  and hence $\overline{\tau}_2=p_{\max}$.

	Next choose some $\tau<\overline{\tau}_2\wedge q$. Fubini's theorem implies for $\xi\in (0,1)$ that 
	\begin{align*}
		\int_{1<\abs{\beta}\le\xi^{-1}}\abs{\beta}^q\,\overline{\rho}(\d \beta)
		&=\int_0^\infty \overline{\rho}\big(\{\abs{\beta}^q\1_{\{1<\abs{\beta}\le\xi^{-1}\}}>t\}\big)\,\d t\\
		&=\int_0^1 \overline{\rho}\big(\{1<\abs{\beta}\le\xi^{-1}\}\big)\,\d t+
		\int_1^{\xi^{-q}} \overline{\rho}\big(\{t^{\frac{1}{q}}<\abs{\beta}\le\xi^{-1}\}\big)\,\d t\\
		&\le \overline{\rho}\big(\{1<\abs{\beta}\le\xi^{-1}\}\big) + \int_1^{\xi^{-q}} \overline{\rho}\big(\{t^{\frac{1}{q}}<\abs{\beta}\}\big)\,\d t\\
		&\le\overline{\rho}(\R) + C \int_1^{-\xi^q} t^{-\frac{\tau}{q}}\,\d t
		\lesssim 1+\xi^{\tau-q}.
	\end{align*}
It follows  $\xi^{q-\tau}\int_{1<\abs{\beta}\le\xi^{-1}}\abs{\beta}^q\,\rho(\d \beta)<\infty$ for all $\xi<1$, 
implying $\tau\le\overline{\tau}_1^{(q)}$. As $\tau<\overline{\tau}_2\wedge q$ is arbitrary we have shown that $\overline{\tau}_1^{(q)}\ge \overline{\tau}_2\wedge q$. As $\overline{\tau}^{(q)}=\overline{\tau}_1^{(q)}\wedge\overline{\tau}_2$ and $\overline{\tau}_2=p_{\max}$, we  conclude $\overline{\tau}^{(q)}=p_{\max}\wedge q$.
\end{proof}

The following result gives conditions such that a hedgehog process is induced by a L\'evy process in a certain Besov space. 
The critical value will be the parameters $\overline{\tau}^{(\min\{p,2\})}$ and $\underline{\tau}^{(\max\{p,2\})}$. 
\begin{theorem}
	\label{thm:Hedgehog}
	Let $L$ be defined by \eqref{eq.hedgehog} and let $p>1$ and $(s,w)\in E_{p}$.
	Define
	\begin{align*}
		q_{\min}:=\inf\left\lbrace q> 0\colon \sum_{k=1}^\infty\norm{a_ke_k}_{\besov{p}{s}{w}}^q<\infty\right\rbrace.
	\end{align*}
	Then, 
	\begin{itemize}
		\item[{\rm (1)}] $L$ is induced by a L\'evy process in $\Besov{p}{s}{w}$ if 
		one of the following is satisfied: 
		\begin{itemize}
			\item[{\rm (i)}] $(s,w)\in R_{p}$;
			\item[{\rm (ii)}] $(s,w)\in R_{p}^c\;$ and  $\;q_{\min}<\overline{\tau}^{(\min\{p,2\})}$.
		\end{itemize}
		\item[{\rm (2)}] $L$ is not induced by a process in $\Besov{p}{s}{w}$  if:
        \begin{itemize}
        	\item[\,] $(s,w)\in R_{p}^c\;$ and  $\;q_{\min}>\underline{\tau}^{(\max\{p,2\})}$.
        \end{itemize}		
	\end{itemize} 
\end{theorem}

\begin{proof}
	Part (1): the first alternative condition follows from the 0-Radonification in Theorem \ref{thm:Radonification}. To show the second alternative, we note that $\Besov{p}{s}{w}$ is of type $\min\{p,2\}$ by the isometry with $\ell^p(\mathbb{W}^d)$ defined in \eqref{eq:Isomorphism}. By Proposition 7.1.16 in \cite{Hytonen2017}, if $q\le \min\{p,2\}$ is such that $E\abs{\ell_1}^q<\infty$ then $\big(\norm{a_k e_k}_{\besov{p}{s}{w}}\big)_{k\in\N}\in\ell^q(\R)$ implies that $\sum_{k\in\N} a_ke_k\ell_k(1)$ converges in $\Besov{p}{s}{w}$ in $q$-th mean. This limit is a L\'evy procss in $\Besov{p}{s}{w}$ which induces $L$. 
	Finally we note that, by Proposition~\ref{prop:HedgeIndexCons}, $\ell_1$ has moments of any order smaller than $\overline{\tau}^{(q)}$ for any $q\le 2$.
	
	Part (2): by Corollary \ref{cor:HedgehogLevyMeasure}, $L$ is not induced by a L\'evy process in $\Besov{p}{s}{w}$ if
	\begin{align}
	\label{eq:HedgehogPart1}
		\sum_{k=1}^\infty \int_{\R} \left(\norm{a_k e_k}_{\besov{p}{s}{w}}^{\max\{2,p\}}\abs{\beta}^{\max\{2,p\}}\wedge 1\right)\,\rho({\rm d}\beta)=\infty.
	\end{align}
Due to the hypothesis,  we can choose $q> \underline{\tau}^{(\max\{2,p\})}$
such that $\big(\norm{a_k e_k}_{\besov{p}{s}{w}}\big)_{k\in\N}\notin\ell^{q}(\R)$. 
The very definition of  $\underline{\tau}^{(\max\{2,p\})}$ guarantees that there exists a constant $K$ such that, for large enough $k$, we have 
	\begin{align*}
		\int_{B_{\R}^c} \big(\norm{a_k e_k}_{\besov{p}{s}{w}}^{\max\{2,p\}}\abs{\beta}^{\max\{2,p\}}\wedge 1\big)\,\rho({\rm d}\beta) 
		&\ge K\norm{a_k e_k}_{\besov{p}{s}{w}}^{q},
	\end{align*}
which establishes \eqref{eq:HedgehogPart1}. 
\end{proof}
\begin{remark}
\label{rem:AdverseBetaGamma}
We may conclude that $(s,w)\in R_p$ implies $q_{\min}\le\underline{\tau}^{(\max\{2,p\})}$, as otherwise Part (1) and (2) of Theorem \ref{thm:Hedgehog} would contradict. This equality can also be proven analytically. 
%Furthermore, one can show that $(s,w)\in E_p$ implies that $\big(\norm{e_k}_{\besov{p}{s}{w}}\big)_{k\in\N}\in\ell^{\infty}(\N)$ for any orthonormal basis $(e_k)_{k\in\N}$. \m{how? why is this worth to mention?}
\end{remark}

The first two examples we present show that, in the case each $\ell_k$ has moments of all orders, the critical summability is of a particularly simple form.
\begin{example}
	Let $L$ be a cylindrical L\'evy process of the form \eqref{eq.hedgehog} with $\rho=\delta_1$; thus each of the $\ell_k$ is a Poisson process with unit intensity and Condition \eqref{eq:HedgehogWeightSum} is satisfied for  $(a_k)_{k\in\N}\in\ell^\infty(\N)$. As $\rho$ has moments of all orders, Proposition \ref{prop:HedgeIndexCons} implies $\overline{\tau}^{(q)}=\underline{\tau}^{(q)}=q$ for each $q\in\Rp$. Thus, the critical summability needs to satisfy $q_{\min}<p\wedge 2$ for inclusion and $q_{\min}>p\vee 2$ for exclusion.
\end{example}

\begin{example}
	Let $\rho(\d\beta)=\1_{\{\beta\neq 0\}}\abs{\beta}^{-\zeta}e^{-\abs{\beta}}\,\d\beta$ for some $\zeta\in(0,3)$, this gives rise to tempered stable processes. For $\zeta=1$ this gives the symmetric Gamma process and for $\zeta=\tfrac{3}{2}$ this gives the symmetric inverse Gaussian process. 
	For $q>\zeta-1$ we have
	\begin{align*}
		\int_{\R}\abs{\beta}^{q-\zeta}e^{-\abs{\beta}}\,\d \beta=2\Gamma(q-\zeta+1)<\infty;
	\end{align*}
	 we conclude that $\rho$ has moments of all orders and thus again we have $\overline{\tau}^{(q)}=\underline{\tau}^{(q)}=q$ for each $q\in\Rp$. Thus, the critical summability again needs to satisfy $q_{\min}<p\wedge 2$ for inclusion and $q_{\min}>p\vee 2$ for exclusion.
\end{example}

Next we examine the symmetric-$\alpha$-stable case, where the limits on moments comes into play.
\begin{example}\label{ex.stable}
Let $\rho(\d \beta)
=\1_{\{\beta\neq 0\}} \abs{\beta}^{-1-\alpha}\d \beta$ for some $\alpha\in (0,2)$. Condition \ref{eq:HedgehogWeightSum} is satisfied if and only if $(a_k)_{k\in\N}\in\ell^{2\alpha/(2-\alpha)}(\R)$; see Example 4.5 in \cite{Riedle2015}. It follows by direct computation that $\overline{\tau}^{(q)}=\underline{\tau}^{(q)}=q\wedge\alpha$ for each $q\in\Rp$. 

Let $p>1$ and $(s,w)\in E_{p}$.
In this case, we obtain the following dichotomy in the critical regime $(s,w)\in R_{p}^c\cap E_{p}$: 
\begin{itemize} 
	\item[{\rm (a)}]  $L$ is induced by a L\'evy process $Y$ in $\Besov{p}{s}{w}$  if  $q_{\min}<p\wedge\alpha$;
 	\item[{\rm (b)}]  $L$ is not induced by a process in $\Besov{p}{s}{w}$  if  $q_{\min}>\alpha$.
\end{itemize}
\end{example}

The following example gives a construction whereby $\underline{\tau}^{(q)}\neq \overline{\tau}^{(q)}$.
\begin{example}
	Let $\alpha_1\in(1,2)$ and $\alpha_2\in(\alpha_1,2)$.
	Let $\rho$ be given by 
	\begin{align*}
		\rho(\d \beta)=\sum_{k=0}^\infty\left(\1_{\{(2k,2k+1]\}}(\beta)\abs{\beta}^{-1-\alpha_1}\,\d \beta + \1_{\{(2k+1,2k+2]\}}(\beta)\abs{\beta}^{-1-\alpha_2}\,\d \beta\right).
	\end{align*}
	Then direct calculations show that $\overline{\tau}^{(q)}=q\wedge\alpha_1$ for each $q\in\Rp$ and $\underline{\tau}^{(q)}=\alpha_2$ for each $q\ge 2$.
	Let $p>1$ and $s,w\in E_p$.
	In this case, we obtain for $(s,w)\in R_p^c\cap E_p$: 
	\begin{itemize} 
	\item[{\rm (a)}]  $L$ is induced by a L\'evy process $Y$ in $\Besov{p}{s}{w}$  if  $q_{\min}<p\wedge\alpha_1$;
	\item[{\rm (b)}]  $L$ is not induced by a process in $\Besov{p}{s}{w}$  if  $q_{\min}>\alpha_2$.
	\end{itemize}		
\end{example}

\begin{example}
	Let $\alpha\in(0,2)$ and let $\rho$ be given by 
	\begin{align*}
		\rho(\d \beta)=\1_{\{\beta\neq 0\}} \abs{\beta}^{-1-\alpha}v(\abs{\beta})\d \beta,
	\end{align*}
	where $v$ is a slowly varying function; see e.g.\ Definition 1.2.1 in \cite{Bingham1989}.
	An application of Proposition 1.5.10 in \cite{Bingham1989} shows that $\overline{\tau}^{(q)}=q\wedge\alpha$ for each $q\le 2$. However, it is known, see  \cite[p.16]{Bingham1989}, that there exist slowly varying functions $v$ such that $\liminf_{\beta\to\infty} v(\beta)=0$ and $\limsup_{\beta\to\infty} v(\beta)=\infty$.  Thus, we cannot in general improve on the bound $q\wedge\alpha\le\underline{\tau}^{(q)}\le q$ for this class of processes, which form a subclass of  subexponential L\'evy processes.
\end{example}

%%%%%%%%%%%%%%%%%%%%%
\subsubsection{Hedgehog process defined on wavelet basis}
We may further analyse the symmetric $\alpha$-stable case by selecting an admissible wavelet basis of $\Besov{p}{s}{w}$ as the orthonormal basis of $L^2(\R^d)$ in the representation \eqref{eq.hedgehog} of the hedgehog cylindrical L\'evy process. 
This will allow us to construct counterexamples required to complete the proof of Theorem \ref{thm:Radonification}.

Let $\Psi=\{\Psi_m^{j,G}\colon (j, G, m)\in\mathbb{W}^d\}$ be an admissible basis for $\Besov{p}{s}{w}$
    for some $p>1$ and $(s,w)\in E_p$ and let $(\ell_m^{j,G})_{(j,G,m)\in\mathbb{W}^d}$ be a family of independent identically distributed canonical $\alpha$-stable processes for some $\alpha \in (0,2)$, i.e.  $\rho(\d x)=\1_{\{ x\neq 0\}} \abs{x}^{-1-\alpha}\d x$. We consider a cylindrical L\'evy process $L$ of the form 
   \begin{align}\label{eq.hedgehog-Psi}
   	L(t)f=\sum_{j,G,m} \Scapro{f}{\Psi_m^{j,G}} a_m^{j,G} \ell_m^{j,G}
   	\qquad\text{for all }f\in L^2(\R^d), \, t\ge 0.
   \end{align}
As in Example \ref{ex.stable}, Condition \eqref{eq:HedgehogWeightSum} is satisfied if $\big(a_m^{j,G}\big)_{j,G,m}\in\ell^{\frac{2\alpha}{2-\alpha}}(\mathbb{W}^d)$.

The following Proposition allows us to determine sharp boundaries for each $p>1$ of the 0-Radonification region $R_p$ of the $(s,w)$ plane in the parameter space defining the weighted Besov spaces.
\begin{proposition}
\label{prop:AlphaStable}
	Let $p>2$ and $(s,w)\in E_p\setminus R_p$. Then for any $\alpha\in(0,2)$ there exists a sequence $\big(a_m^{j,G}\big)_{j,G,m}\in\ell^{\frac{2\alpha}{2-\alpha}}(\mathbb{W}^d)$ such that $L$ as constructed in \eqref{eq.hedgehog-Psi} is not induced by a process in $\Besov{p}{s}{w}$.
\end{proposition}
We first state an intermediate result on the summability of the Besov space weights.
\begin{lemma}
\label{lem:WeightSum}
 Let $\omega_m^j=\omega_m^j(p,s,w)$ be the wavelet weight constants for $\Besov{p}{s}{w}$ for some $p>0$, $s<\tfrac{d}{p}-\tfrac{d}{2}$ and $w<0$. Then $\big(\omega_m^j\big)_{j,G,m}\in\ell^k(\mathbb{W}^d)$  for some $k>0$ if and only if
	\begin{align*}
		k>\max\left\lbrace -\frac{d}{w}, \frac{2dp}{2d-dp-2ps}\right\rbrace.
	\end{align*}
\end{lemma}

\begin{proof}
	We must assess the convergence of
	\begin{align*}
		\sum_{j\ge 0}2^{jk(s-\frac{d}{p}+\frac{d}{2})}\sum_{G\in G^j}\sum_{m\in\Z^d}(1+2^{-2j}\abs{m}^2)^{\frac{kw}{2}}.
	\end{align*}
	We first consider $	S_j:=\sum_{m\in\Z^d}(1+2^{-2j}\abs{m}^2)^{\frac{kw}{2}}$.
	We have $S_j<\infty$ for each $j$ if and only if $kw<-d$ (see the proof of \cite[Th.~3]{Fageot2016}) which gives the first term in the maximum above, recalling that $w<0$.
	If $S_j<\infty$ for each $j$, then $S_j$ is asymptotically $\O(2^{jd})$ as $j\to\infty$ according to \cite[Th.~3]{Fageot2016}, and thus
	\begin{align*}
		\sum_{j\ge 0}2^{jk(s-\frac{d}{p}+\frac{d}{2})}\sum_{G\in G^j}S_j
		&= 2^dS_0 +(2^d-1)\sum_{j\ge 1}2^{jk(s-\frac{d}{p}+\frac{d}{2})}S_j,
	\end{align*}
	which is finite if and only if $k(s-\tfrac{d}{p}+\tfrac{d}{2})<-d$. As  $k>0$ and $s-\tfrac{d}{p}+\tfrac{d}{2}<0$ this condition is equivalent to $k>\tfrac{-d}{s-\frac{d}{p}+\frac{d}{2}}$, which completes the proof.
\end{proof}

\begin{proof}[Proof of Proposition \ref{prop:AlphaStable}]
	We note that $\norm{\Psi_m^{j,G}}_{\besov{p}{s}{w}}=\omega_m^j$, where $\omega_m^j=\omega_m^j(p,s,w)$ are the weight constants for $\Besov{p}{s}{w}$, and we have $\big(\omega_m^j\big)_{j,G,m}\in\ell^\infty(\mathbb{W}^d)$ as $(s,w)\in E_p$. 
	
	For $0<q<\tfrac{2\alpha}{2-\alpha}$, the sum
	\begin{align*}
		\sum_{j,G,m}\norm{a_m^{j,G}\Psi_m^{j,G}}_{\Besov{p}{s}{w}}^q
		=\sum_{j,G,m}\abs{a_m^{j,G}\omega_m^{j}}^q
	\end{align*}
	is finite for every $\big(a_m^{j,G}\big)_{j,G,m}\in\ell^{\frac{2\alpha}{2-\alpha}}(\mathbb{W}^d)$ if and only if
	\begin{align}
	\label{eq:AlphaWeightSum}
		\big(\abs{\omega_m^j}^q\big)_{j,G,m}\in\big(\ell^{\frac{2\alpha}{q(2-\alpha)}}(\mathbb{W}^d)\big)^\ast=\ell^{\frac{2\alpha}{2\alpha+\alpha q-2q}}(\mathbb{W}^d).
	\end{align}
	Since $p>2$, the assumption $(s,w)\in E_p$  implies $s<\tfrac{d}{p}-\tfrac{d}{2}$ and $w\le 0$ according to Proposition~\ref{prop:LqEmbeddings}. For the case $w=0$, we note that $\sum_{j,G,m}(\omega_m^j)^k=\infty$ for any $k>0$, and thus there exists $\big(a_m^{j,G}\big)_{j,G,m}\in\ell^{\frac{2\alpha}{2-\alpha}}(\mathbb{W}^d)$ such that $\sum_{j,G,m}\abs{a_m^{j,G}\omega_m^{j}}^\alpha=\infty$.
	We continue to consider the case $w<0$.
	By applying Lemma \ref{lem:WeightSum}, we see that \eqref{eq:AlphaWeightSum} is satisfied if and only if
	\begin{align*}
		\frac{2\alpha q}{2\alpha+\alpha q-2q}>\max\left\lbrace -\frac{d}{w}, \frac{2dp}{2d-dp-2ps}\right\rbrace.
	\end{align*}
	As $q<\tfrac{2\alpha}{2-\alpha}$, 
	we have $2\alpha+\alpha q+2q>0$, and thus
	we have
	\begin{align}
	\label{eq:AlphaHedgehog1}
		\frac{2\alpha q}{2\alpha+\alpha q-2q}>-\frac{d}{w}
		\Leftrightarrow
%		2\alpha q w<-d(2\alpha+\alpha q-2q)
%		\Leftrightarrow
		q>\frac{2\alpha d}{2d-\alpha d- 2\alpha w},
	\end{align}
	where we note that $2d-\alpha d- 2\alpha w>0$ as $\alpha<2$ and $w<0$.
	Furthermore, as $s<\tfrac{d}{p}-\tfrac{d}{2}$ we have $2d-dp-2ps>0$ and so
	\begin{align}
	\label{eq:AlphaHedgehog2}
		\frac{2\alpha q}{2\alpha+\alpha q-2q}>\frac{2dp}{2d-dp-2ps}
		\Leftrightarrow
%		\alpha q(2d-dp-2ps)>dp(2\alpha+\alpha q-2q)
%		\Leftrightarrow
		q>\frac{\alpha dp}{\alpha d-\alpha dp-\alpha ps + dp}
	\end{align}
	where we have $\alpha d-\alpha dp-\alpha ps + dp>dp(1-\tfrac{\alpha}{2})>0$.
	Taking $q=\alpha$, we see that there exists $\big(a_m^{j,G}\big)_{j,G,m}\in\ell^{\frac{2\alpha}{2-\alpha}}(\mathbb{W}^d)$ such that $\sum_{j,G,m}\abs{a_m^{j,G}\omega_m^{j}}^\alpha=\infty$ when either
	$w\ge -\frac{d}{2}$ by \eqref{eq:AlphaHedgehog1}, or
	$s\ge -d+\frac{d}{p}$ by \eqref{eq:AlphaHedgehog2}. 
	
	By referring to the conditions shown in Example \ref{ex.stable}, it follows that if $L$ in \eqref{eq.hedgehog-Psi} is constructed using such a sequence $\big(a_m^{j,G}\big)_{j,G,m}$ with $\sum_{j,G,m}\abs{a_m^{j,G}\omega_m^{j}}^\alpha=\infty$, the summability index $q_{\min}$ as defined in Theorem \ref{thm:Hedgehog} has $q_{\min}>\alpha$, and so $L$ is not induced by a process in $\Besov{p}{s}{w}$.
%	, for $p>2$, if $(s,w)\notin R_{p}=(-\infty,-d+\tfrac{d}{p})\times(-\infty,-\tfrac{d}{2})$ then the embedding $L^2(\R^d)\hookrightarrow\Besov{p}{s}{w}$ is not 0-Radonifying.
\end{proof}

\end{document}